\documentclass[a4wide,13pt]{article}

\usepackage{geometry}                

\usepackage{graphicx}
\usepackage{authblk}
\usepackage{amssymb}
\usepackage{epstopdf}
\usepackage{rotating}
\usepackage[sans]{dsfont}
\usepackage[T1]{fontenc}
\usepackage[english]{babel}
\usepackage{latexsym}
\usepackage{mathrsfs}
\usepackage{amscd}
\usepackage{sidecap}
\usepackage{graphicx}
\usepackage{color}
\usepackage{algorithm}

\usepackage{amsmath}
\usepackage{amsfonts}
\numberwithin{equation}{section}
	\usepackage{enumerate}
\usepackage{amsthm}
\usepackage[numbers,sort]{natbib}

\usepackage{array}
\newcolumntype{P}[1]{>{\centering\arraybackslash}p{#1}}
\newcolumntype{M}[1]{>{\centering\arraybackslash}m{#1}}

     \graphicspath{{pics/}}

\def\bigF{ {\mathcal {P}}}

\def\bigK{ {\mathcal {K}}}
\def\Exp{\mathbb{E}}
\def\CLambda{\Psi}
\def\U{ {\mathcal {U}}}
\def\P{ {\mathcal {M}}}

\def\QF{\mathcal{P}}

\def\mrho{\mu}
\newcommand{\R}{\mathbb R}
\newcommand{\N}{\mathbb N}

\newcommand{\epsi}{{\varepsilon}}
\newcommand{\sidecap}[1]{ {\begin{sideways}\parbox{3.cm}{\centering #1}\end{sideways}} }
\def\be#1\ee{\begin{equation}#1\end{equation}}

\newtheorem{thm}{Theorem}[section] 
\newtheorem{definition}[thm]{Definition}
 
\newtheorem{lem}[thm]{Lemma}

\newtheorem{defn}{Definition}[section] 

\newtheorem{remark}{Remark}[section] 

\newcommand{\bq}{\begin{equation}}
\newcommand{\eq}{\end{equation}}

\def\bqa{\begin{eqnarray}}
\def\eqa{\end{eqnarray}}

\def\e{\epsilon}

\newcommand{\bd}{\begin{displaymath}}
\newcommand{\ed}{\end{displaymath}}
\newcommand{\ba}{\begin{eqnarray}}
\newcommand{\ea}{\end{eqnarray}}

\def\N{\mathbb{N}}
\def\R{\mathbb{R}}

\def\pa{\partial}

\def\epsi{ \varepsilon}

\newcommand{\lt}{\left}
\newcommand{\rt}{\right}
\newcommand{\om}{\Omega}
\newcommand{\lal}{\langle}
\newcommand{\ral}{\rangle}
\newcommand{\ls}{\lesssim}

\newcommand{\mc}{\mathcal{C}}
\newcommand{\F}{\mathcal{F}}
\newcommand{\W}{\mathcal{W}}
\newcommand{\weakto}{\stackrel{*}{\rightharpoonup}}
\newcommand{\mt}{\mathcal{T}}

\begin{document}
\title{Mean field  control hierarchy}

\author[1]{Giacomo Albi \thanks{giacomo.albi@ma.tum.de, address: Boltzmannstr. 3, D-85748 Garching bei M\"unchen, Germany}}
\author[1]{Young-Pil Choi  \thanks{ychoi@ma.tum.de address: Boltzmannstr. 3, D-85748 Garching bei M\"unchen, Germany}}
\author[1]{Massimo Fornasier  \thanks{massimo.fornasier@ma.tum.de, address: Boltzmannstr. 3, D-85748 Garching bei M\"unchen, Germany}}
\author[2]{Dante Kalise \thanks{dante.kalise@oeaw.ac.at, address: RICAM, Altenbergerstr. 69, A-4040 Linz, Austria}}
\affil[1]{Department of Mathematics, TU M\"unchen}
\affil[2]{RICAM, Austrian Academy of Sciences, Linz}

\maketitle

\begin{abstract}
In this paper we model the role of a government of a large population as a mean field optimal control problem. Such control problems are constrainted  by a PDE of continuity-type, governing the dynamics of the probability distribution of the agent population. We show the existence of mean field optimal controls both in the stochastic and deterministic setting. We derive rigorously the first order optimality conditions useful for numerical computation of mean field optimal controls. We introduce a novel approximating hierarchy of sub-optimal controls based on a Boltzmann approach, whose computation requires a very moderate numerical complexity with respect to the one of the optimal control. We provide numerical experiments for models in opinion formation comparing the behavior of the control hierarchy.
\end{abstract}

\section{Introduction}
\maketitle
Self-organization in  social interactions 
is a fascinating mechanism, which inspired 
the mathematical modeling of multi-agent interactions towards formation of coherent global behaviors, with applications in the study of biological,  social, and economical  phenomena.  
Recently there has been a vigorous development of literature in applied mathematics and physics describing collective behavior of 
multiagent systems \cite{CuckerDong11,cucker-mordecki,CS,GC04, MR2000132,KeMinAuWan02,vicsek}, towards modeling phenomena in biology, such as cell aggregation and motility \cite{CDFSTB03,kese70,KocWhi98,be07},  coordinated animal motion \cite{BCCCCGLOPPVZ09,MR2507454,ChuDorMarBerCha07,crpito10,CouFra02,CKFL05,CS,Niw94,PE99,ParVisGru02,Rom96,TonTu95,YEECBKMS09}, coordinated human \cite{crpito11,CucSmaZho04,MR2438215} and synthetic agent behavior and interactions, such as 
cooperative robots \cite{ChuHuaDorBer07,LeoFio01,PerGomElo09,SugSan97}. As it is very hard to be exhaustive in accounting all the developments of this very fast growing
field, we refer to \cite{CCH13,CCP16, CHL16, cafotove10,viza12}  for recent surveys.\\
Two main mechanisms are considered in such models to drive the dynamics. The first, which takes inspiration, e.g., from physics laws of motion, is {based on} binary forces encoding observed ``first principles'' 
of biological, social, or economical interactions.
Most of these models start from particle-like systems, borrowing a leaf from Newtonian physics, by including fundamental ``social interaction'' forces within classical systems of 1st or 2nd order equations. In this  paper we mix general principles with concrete modeling instances to encounter  the need of both a certain level of generality and  to provide immediately  a concrete applications. Accordingly, we  consider here mainly large particle/agent systems of form:
\begin{align}\label{eq:HK}
 dx_i = \lt(\frac{1}{N}\sum_{j=1}^N P(x_i,x_j)(x_j-x_i)\rt)dt + \sqrt{2\sigma} \,dB_i^t, \qquad i=1,\ldots,N, \quad t > 0,
\end{align}
where $P(\cdot,\cdot)$ represents the communication function between agents $x_i \in \mathbb R^d$ and $B_i^t$ is a $d$-dimensional Brownian motion.

  The second mechanism, which we do not address in detail here, is based on evolutive 
games, where the dynamics is driven by the simultaneous optimization of costs by the players, perhaps subjected to selection, from game theoretic models of evolution \cite{zbMATH01169593} to 
mean field  games, introduced in \cite{lali07} and independently  under the name
Nash Certainty Equivalence (NCE) in \cite{HCM03}, later greatly popularized, e.g., within consensus problems, for instance in \cite{NCM10,NCM11}.

The common viewpoint of these branches of mathematical modeling of multi-agent systems is that the dynamics are based on the {\it free} interaction of the agents  or {\it decentralized} control. The wished phenomenon
to be described is their self-organization in terms of the formation of complex  macroscopic patterns. 

One fundamental goal of these studies is in fact to reveal the possible relationship between the simple binary forces acting at individual level, being the ``first principles'' of
social interaction or the game rules, and the potential emergence of a global behavior in the form of specific patterns.  

For instance one can use the model in \eqref{eq:HK}, for $d=1$ and  $x_i\in I=[-1,1]$, a bounded interval, to formulate classical opinion models, where $x_i$ represents an opinion in the continuos set between two opposite opinions $\{-1,1\}$. According to the choice of the communication function $P(\cdot,\cdot)$, consensus can emerge or not, and different studies have been made in order to enforce the emergence of a global consensus, \cite{ABCK15,AHP,APZa,T,DHL14}. 
The mathematical property for a system to form patterns is actually its persistent {\it compactness}. 
There are actually several mechanisms of promotion of compactness to yield eventually self-organization. In the recent paper  \cite{MT13}, for instance, the authors name the {\it heterophilia}, i.e., the tendency to bond more with those who are ``different'' rather than those who are similar, as a positive mechanism in consensus models to reach accord. However also in {\it homophilious} societies  influenced by more local interactions, global self-organization towards consensus can be expected as soon as enough initial coherence is given. At this point, and perhaps reminiscently  of biblic stories from the Genesis, one could enthusiastically argue {\it ``Let us give them good rules and they will find their way!''} Unfortunately, this is not true, at all. In fact, in homophilious regimes there are plenty of situations where patterns will not spontaneously form. In Section \ref{sec:num} below we mathematically demonstrate with a few simple numerical examples the incompleteness of the self-organization paradigm, and 
we refer to \cite{bofo16} for its systematic discussion. Consequently, we propose to amend it by allowing possible external interventions in form of {\it centralized} controls. The human society calls them {\it government}.

The general idea consists in considering dynamics of the form
\begin{align}\label{eq:HKu}
d x_i =\lt(\frac{1}{N}\sum_{j=1}^N P(x_i,x_j)(x_j-x_i)\rt)dt + f_i\,dt + \sqrt{2\sigma} \,dB_i^t, \qquad i=1,\cdots,N, \quad t > 0,
\end{align}
where the control $f = (f_1,\ldots,f_N)$ minimizes a given functional $J(x,f)$. As an example we can 
consider the following variational formulation
\begin{align}\label{eq:funcJu}
f = \arg\min_{g\in\U} J(x,g):=  \mathbb E \left [ \int_0^T\frac{1}{N}\sum_{i=1}^N\left(\frac{1}{2}|x_i-x_d|^2 + \gamma\CLambda(g_i)\right)\,dt \right ],
\end{align}
where $x_d$ represents a target point, $\gamma$ is the penalization parameter of the control $g$, which is chosen among the admissible controls in $\U$, and $\CLambda:\R^d \to \R_+ \cup \{0 \}$ is a convex function. The choice of this particular cost function, and especially of the term $\int_0^T \frac{1}{2}\int |x-x_d|^2\mrho(x,t)\,dx$ is absolutely arbitrary. It is consistent with our wish of mixing general statements with instances of applications, and the cost function is so given to provide immediately a specific instance of application oriented to opinion consensus problems.
Similar models as \eqref{eq:funcJu} have been studied recently also for the flocking dynamics in \cite{CFPT,FS13,BF,APa} and one can of course consider many more instances, as soon as one ensures enough continuity of the cost, see, e.g., \cite{FS13}.

As the number of particles $N \to \infty$, the finite dimensional optimal control problem with ODE constraints \eqref{eq:HKu}-\eqref{eq:funcJu} converges to the following mean field optimal control problem \cite{befrph13,Achdou2016,FS13}:
\bq\label{main_eq}
\pa_t \mrho + \nabla \cdot \lt(\lt(\QF[\mrho] + f\rt)\mrho \rt) = \sigma \Delta \mrho,
\eq
where the interaction force $\mathcal P$ is given by
\begin{equation}\label{eq:kernelP}
\QF[\mrho](x) = \int P(x,y)(y-x)\mrho(y,t)\,dy
\end{equation}
and the solution $\mrho$ is controlled by the minimizer of the cost functional
\bq\label{main_func}
J(\mrho,f) = \int_0^T\left(\frac{1}{2}\int |x-x_d|^2\mrho(x,t)\,dx + \gamma\int \CLambda(f)\mrho(x,t)\,dx\right)\,dt.
\eq
To a certain extent, the mean field optimal control problem  \eqref{main_eq}-\eqref{main_func} can be viewed as a generalization of optimal transport problems \cite{zbMATH01419740} for which the term $P \equiv 0$, the term $\int_0^T \frac{1}{2}\int |x-x_d|^2\mrho(x,t)\,dx$ does not appear in the cost, and final conditions are given.
Differently from mean field games \cite{lali07}  the goal here is not to derive the equlibria of a multi-player game, rather to compute  mean field optimal government strategies for a population so large that the curse of dimensionality would otherwise prohibit numerical solutions. 
The mean field optimal control problem \eqref{main_eq}-\eqref{main_func} provides an artificial  confinement vector field $f$, inducing the right amount of compactness to have global convergence to steady states (pattern formation). {\it Local} convergence towards, e.g., to global Maxwellians, is provided for certain second order mean field-type of equations in \cite{Choi16, DFT10}. Hence, our results can be also interpreted as an external model perturbation to induce {\it global} stability. 
\\

In this paper we provide a friendly introduction to mean field optimal controls of the type \eqref{main_eq}-\eqref{main_func}, showing their main analytical properties and furnish a simple route to their numerical solutions, which we call ``the control hierarchy''. Although some of the results contained in this paper are certainly also derived elsewhere, see, e.g., \cite{befrph13,FS13}, we made an effort to present them in a simplified form as well as providing rigorous derivations.

In particular, in Section \ref{sec:existence}, we show existence of mean field optimal controls for first order models  in case of both stochastic and deterministic control problems.
 We also derive rigorously in Section \ref{sec:firstorderopt} the corresponding first order optimality conditions, resulting in a coupled system of forward/backward time-dependent PDEs. The forward equation is given by \eqref{main_eq}, while the backward one is a nonlocal integro-differential advection-reaction-diffusion equation. The presence of nonlocal interaction terms in form of integral functions is another feature, which distinguishes  mean field optimal control problems from classical mean field games \cite{lali07} and optimal transport problems \cite{zbMATH01419740}, where usually $P \equiv 0$. The nonlocal terms pose additional challenges in the numerical solution, which are subject of recent studies \cite{CJ09}.
 
 Although  mean field optimal controls are designed to be independent of the number $N$ of agents to provide a way to circumvent the course of dimensionality of $N\to \infty$, still their numerical computation needs to be realized by solving the first-order optimality conditions. The complexity of their solution depends on the intrinsic dimensionality $d$ of the agents, which is affordable only at moderate dimensions (e.g., $d \leq 3$). For this reason, in Section \ref{sec:Boltz} we approach the solution of the mean field optimal control, by means of a novel hierarchy of suboptimal controls, computed by a Boltzmann approach: first one derives a control for a system of two representative particles, then one plugs it into a collisional operator considering the statistics of the interactions of a distribution of agents, and finally one performs a {\it quasi-invariant limit} to approximate the PDE of continuity-type, governing the dynamics of the probability distribution of the agent population.
For the two particle system considered in the first step of the Boltzmann approach above, we propose two suboptimal controls stemming from the binary Boltzmann approach:
 the first level is given by an instantaneous {\it model predictive control}  on two interacting agents - we shall call this control {\it instantaneous control} (IC) -, while the second stems from the solution of the binary optimal control problem by means of the Bellman dynamical programming  principle - we shall call this control {\it finite horizon control} (FH) - . These two controls have the advantage that the complexity of their computation is dramatically reduced with respect to the mean field optimal control (OC) in its full glory, still retaining their ability to induce government of the population. We describe in detail how they can be efficiently numerically computed.
 In Section \ref{sec:num} we  provide simple numerical approaches, easily implementable, for solving one-dimensional mean field optimal control problems of the type \eqref{main_eq}-\eqref{main_func}. We eventually numerically compare the control hierarchy with the mean field optimal control in a model of opinion formation and we show the quasi-optimality of the Boltzmann-Bellman (FH) control. 

%
%
%
%
%
\section{Existence of mean field optimal controls}\label{sec:existence}

%
%
%
%
%
\subsection{Deterministic case}
In this section, we study global existence and uniqueness of weak solutions for the equation \eqref{main_eq} in $\R^d$ without the diffusion, i.e., $\sigma=0$, namely
\begin{equation}\label{eq_det}
\pa_t \mrho + \nabla \cdot \lt(\lt(\QF[\mrho] + f\rt)\mrho \rt) = 0,\quad x \in \mathbb R^d, \quad t>0.
\end{equation}
We also investigate the mean field limit of the ODE constrained control problem \eqref{eq:HKu}-\eqref{eq:funcJu} in the deterministic setting. Let us denote by $\P(\R^d)$ and $\P_p(\R^d)$ the sets of all probability measures and the ones with finite moments of order $p \in [1,\infty)$ on $\R^d$, respectively. We first define a notion of weak solutions to the equation to \eqref{eq_det}.

\begin{defn}\label{def_weak0}For a given $T > 0$, we call $\mrho \in \mc([0,T];\P_1(\R^d))$ a weak solution of \eqref{eq_det} on the time-interval $[0,T]$ if for all compactly supported test functions $\varphi \in \mc^\infty_c(\R^d \times [0,T])$,
$$\begin{aligned}
\int_{\R^d} \varphi(x,T)\,\mrho_T(dx) - \int_0^T\int_{\R^d} \lt( \pa_t \varphi + \lt( \QF[\mu_t] + f \rt)\cdot \nabla \varphi\rt)\mrho_t(dx)dt = \int_{\R^d} \varphi_0(x)\,\mrho_0(dx).
\end{aligned}$$
\end{defn}
We also introduce a set of admissible controls $\F_\ell([0,T])$ in the definition below.
\begin{defn}For a given $T$ and $q \in [1,\infty)$, we fix a control bound function $\ell \in L^q(0,T)$. Then $f \in \F_\ell([0,T])$ if and only if
\begin{itemize}
\item[(i)]  $f : [0,T] \times \R^d \to \R^d$ is a Carath\'eodory function.
\item[(ii)] $f(\cdot,t) \in W^{1,\infty}_{loc}(\R^d)$ for almost every $t \in [0,T]$.
\item[(iii)] $|f(0,t)| + \|f(\cdot,t)\|_{\rm{Lip}} \leq \ell(t)$ for almost every $t \in [0,T]$. 
\end{itemize}
\end{defn}

For the existence and mean field limit, we use the topology on probability measures induced by the Wasserstein distance, which is defined by
\[
\W_p(\mu,\nu) := \inf_{\pi \in \Gamma(\mu,\nu)} \lt(\int_{\R^{2d}} |x - y|^p\,\pi(dx,dy) \rt)^{1/p} \quad \mbox{for} \quad p \geq 1 \,\,\mbox{ and }\,\, \mu,\nu \in \P(\R^d),
\]
where $\Gamma(\mu,\nu)$ is the set of all probability measures on $\R^{2d}$ with first and second marginals $\mu$ and $\nu$, respectively. Note that $\P_1(\R^d)$ is a complete metric space endowed with the $\W_1$ distance, and $\W_1$ is equivalently characterized in duality with Lipschitz continuous functions \cite{vi09}.\\

The following result is a rather straightforward adaptation from \cite{FS13} and we shall prove it rather concisely. For more details we address the interested reader to \cite{FS13}, which has been written in a more scholastic and perhaps accessible form.

\begin{thm}\label{thm_ext} Let the initial data $\mrho_0 \in \P_1(\R^d)$ and assume that $\mrho_0$ is compactly supported, i.e., there exists $R > 0$ such that
\[
supp \,\mrho_0 \subset B(0,R),
\]
where $B(0,R) := \{ x \in \R^d : |x| < R\}$. Furthermore, we assume that $P \in W^{1,\infty}(\R^{2d})$. Then, for a given $f \in \F_\ell([0,T])$, there exists a unique weak solution $\mrho \in \mc([0,T];\P_1(\R^d))$ to the equation \eqref{main_eq} with $\sigma=0$. Furthermore, $\mrho$ is determined as the push-forward of the initial measure $\mrho_0$ through the flow map generated by the locally Lipschitz velocity field $\QF[\mrho] + f$. Moreover, if $\mrho^i,i=1,2$ are two such with initial data $\mrho_0^i$ satisfying the above assumption, we have
\[
\W_1(\mrho^1_t,\mrho^2_t) \leq C\W_1(\mrho^1_0,\mrho^2_0) \quad \mbox{for} \quad t \in [0,T],
\]
where $C > 0$ depends only on $\|P\|_{W^{1,\infty}}$, $R$, $T $, and $\|\ell\|_{L^q}$.
\end{thm}
\begin{proof} $\bullet$ {\bf (Existence \& Uniqueness).-} Let $\mrho \in \mc([0,T];\P_1(\R^d))$ with compact support in $B(0,R)$ for some positive constant $R > 0$. Then we can easily show that the interaction force $\QF$ is locally bounded and Lipschitz:
\[
|\QF[\mrho](x)| \leq C(\|P\|_{L^\infty}, R)(1 + |x|),
\]
and
\[
|\QF[\mrho](x) - \QF[\mrho](y)| \leq C(\|P\|_{W^{1,\infty}},R)(1 + |x|)|x- y|.
\]
On the other hand, since  $f \in \F_\ell([0,T])$, we obtain that the vector field $\QF[\mrho] + f$ is also locally bounded and Lipschitz. Then this together with employing the argument in \cite[Theorem 3.10]{CCR11} and existence theory for Carath\'eodory differential equation in \cite{Fil}, we can get the local-in-time existence and uniqueness of weak solutions to the system \eqref{main_eq} with $\sigma=0$ in the sense of Definition \ref{def_weak0}. Note that those solutions exist as long as that solutions are compactly supported. Set 
\[
R(t) := \max_{x,y \in \overline{\mbox{\small supp}(\mrho_t)}}|x - y| \quad \mbox{for} \quad t \in [0,T].
\]
Let us consider the following characteristic $X(t):=X(t;s,x): \R_+ \times \R_+ \times \R^d \to \R^d$:
\bq\label{eq_char}
\frac{d X(t;s,x)}{dt} = \QF[\mrho_t](X(t;s,x),t) + f(X(t;s,x),t) \quad \mbox{for all} \quad t,s \in [0,T],
\eq
with the initial data $X_0 = x \in \R^d$. We notice that characteristic is well-defined on the time interval $[0,T]$ due to the regularity of the velocity field. A straightforward computation yields that for $x,y \in$ supp$(\mrho_0)$
$$\begin{aligned}
&\frac{ d|X(t) - Y(t)|^2}{dt} \cr
&\quad = (X(t) - Y(t)) \cdot \frac{d \lt( X(t) - Y(t)\rt)}{dt}\cr
&\quad \leq |X(t) - Y(t)|\lt| \QF[\mrho_t](X(t),t) - \QF[\mrho_t](Y(t),t)\rt| + |X(t) - Y(t)||f(X(t),t) - f(Y(t),t)|\cr
&\quad \leq  2\|P\|_{L^\infty}|X(t) - Y(t)| \int_{\R^d} |z - X(t)|\mrho(z,t)\,dz + \|P\|_{L^\infty}|X(t) - Y(t)|^2\cr
&\qquad + \|f(\cdot,t)\|_{\rm{Lip}}|X(t) - Y(t)|^2.
\end{aligned}$$
This deduces
\[
\frac{d R(t)}{dt} \leq \lt( 3\|P\|_{L^\infty} + \|f(\cdot,t)\|_{\rm{Lip}}\rt)R(t) \leq \lt( 3\|P\|_{L^\infty} + \ell(t) \rt)R(t),
\]
and
\[
R(t) \leq CR_0 \quad \mbox{for} \quad t \in [0,T],
\]
where $C$ depends only on $T$, $\|P\|_{L^\infty}$, and $\|\ell\|_{L^q}$. Thus, by continuity arguments, we have the global existence of weak solutions. We can also find that for $h \in \mc^\infty_c(\R^d)$
\[
\int_{\R^d} \mrho(x,t)h(x)\,dx = \int_{\R^d} \mrho_0(x) h(X(0;t,x))\,dx \quad \mbox{for} \quad t \in [0,T].
\]
This implies that $\mrho$ is determined as the push-forward of the initial density through the flow map \eqref{eq_char}.

$\bullet$ {\bf (Stability estimate).-} Let $T>0$ and $\mrho^i,i=1,2$ be the weak solutions to the equation \eqref{main_eq} with $\sigma = 0$ obtained in the above. Let $X_i$ be the characteristic flows defined in \eqref{eq_char} generated by the velocity fields $\QF[\mrho^i] + f$, respectively. For a fixed $t_0 \in [0,T]$, we choose an optimal transport map for $\W_1$ denoted by $\mt^0(x)$ between $\mrho^1_{t_0}$ and $\mrho^2_{t_0}$, i.e., $\mrho^2_{t_0} = \mt^0 \# \mrho^1_{t_0}$. It also follows from the above that $\mrho^i_t = X_i(t;t_0,\cdot) \# \mrho^i_{t_0}$ for $t \geq t_0$. Furthermore, we get $\mt^t \# \mrho^1_t = \mrho^2_t$ with $\mt^t = X_2(t;t_0,\cdot,\cdot) \circ \mt^0 \circ X_1(t_0;t,\cdot)$ for $t \in [t_0,T]$. Then we obtain
$$\begin{aligned}
\frac{d^+ \W_1(\mrho^1_t,\mrho^2_t)}{dt}\Big|_{t = t_0+} &\leq \int_{\R^d} \lt|\QF[\mrho^1_{t_0}](X_1(t;t_0,x),t) - \QF[\mrho^2_{t_0}](X_2(t;t_0,\mt^0(x)),t) \rt| \mrho^1_{t_0}(dx)\Big|_{t = t_0+}\cr
&\quad + \int_{\R^d} \lt|f(X_1(t;t_0,x),t)- f(X_2(t;t_0,\mt^0(x)),t) \rt| \mrho^1_{t_0}(dx)\Big|_{t = t_0+}\cr
&=I_1 + I_2,
\end{aligned}$$
where $I_i,i=1,2$ are estimated as follows.
$$\begin{aligned}
I_1 &\leq \int_{\R^{2d}} \lt|P(x,y)(y-x) - P(\mt^0(x),\mt^0(y))(\mt^0(y) - \mt^0(x)) \rt|\mrho^1_{t_0}(dx)\mrho^1_{t_0}(dy) \cr
&\leq \int_{\R^{2d}}|P(x,y) - P(\mt^0(x),\mt^0(y))||y-x|\mrho^1_{t_0}(dx)\mrho^1_{t_0}(dy)  \cr
&\quad + \int_{\R^{2d}}|P(\mt^0(x),\mt^0(y))|\lt(|y-\mt^0(y)| + | x - \mt^0(x)| \rt)\mrho^1_{t_0}(dx)\mrho^1_{t_0}(dy) \cr
&\leq C\|P\|_{W^{1,\infty}}\W_1(\mrho^1_{t_0},\mrho^2_{t_0}), \cr
I_2 & = \int_{\R^d} \lt|f(x,t) - f(\mt^0(x),t) \rt|\mrho^1_{t_0}(dx) \leq \|f(\cdot,t)\|_{\rm{Lip}}\W_1(\mrho^1_{t_0},\mrho^2_{t_0}) \leq \ell(t)\W_1(\mrho^1_{t_0},\mrho^2_{t_0}),
\end{aligned}$$
where we used the fact that $\mrho$ has the compact support for the estimate of $I_1$. We now combine the above estimates together with being $t_0$ arbitrary in $[0,T]$ to conclude 
\[
\frac{d^+ \W_1(\mrho^1_t,\mrho^2_t)}{dt} \leq C\lt(\|P\|_{W^{1,\infty}} + \ell(t) \rt)\W_1(\mrho^1_t,\mrho^2_t), \quad \mbox{for} \quad t \in [0,T].
\]
This completes the proof.
\end{proof}
In Theorem \ref{thm_ext}, we show the global existence and uniqueness of weak solutions $\mrho$ to the equation \eqref{main_eq} with $\sigma = 0$ for a given control $f \in \F_\ell([0,T])$. In the rest of this part, we show the rigorous derivation of the infinite dimensional optimal control problem from the finite dimensional one as $N \to \infty$. Let us recall the finite/infinite dimensional optimal control problems:

\begin{itemize}
\item {\it Finite dimensional optimal control problem:}
\bq\label{cost_finite}
\min_{f \in \F_\ell}J(x,f):= \min_{f \in \F_\ell}\int_0^T\frac{1}{N}\sum_{i=1}^N\left(\frac{1}{2}|x_i-x_d|^2 + \gamma \CLambda(f_i)\right)\,dt,
\eq
where $x_i$ is a unique solution of
\bq\label{eq_finite}
\dot{x}_i =\frac{1}{N}\sum_{j=1}^N P(x_i,x_j)(x_j-x_i) + f_i, \qquad i=1,\cdots,N, \quad t > 0,
\eq
\item {\it Infinite dimensional optimal control problem:}
\bq\label{cost_inf}
\min_{f \in \F_\ell}J(\mrho_t,f):= \min_{f \in \F_\ell}\int_0^T\left(\frac{1}{2}\int_{\R^d} |x-x_d|^2 \,\mrho_t(dx) + \gamma\int_{\R^d} \CLambda(f) \,\mrho_t(dx) \right)\,dt,
\eq
where $\mrho \in \mc([0,T];\P_1(\R^d))$ is a unique weak solution of 
\begin{align}\label{eq_inf}
\begin{aligned}
\pa_t \mrho_t &= \nabla \cdot \lt( \lt(\QF[\mrho_t] + f\rt)\mrho_t \rt), \quad (x,t) \in \R^d \times [0,T],\cr
\QF[\mrho_t](x) &= \int_{\R^d} P(x,y)(y-x)\mrho_t(dy).
\end{aligned}
\end{align}
\end{itemize}
For the convergence from \eqref{cost_finite}-\eqref{eq_finite} to \eqref{cost_inf}-\eqref{eq_inf}, we need a weak compactness result in $\F_\ell$ whose proof can be found in \cite[Corollary 2.7]{FS13}.
\begin{lem}\label{lem_comp_f}Let $p \in (1,\infty)$. Suppose that $(f_j)_{j \in \N} \in \F_\ell$ with $\ell \in L^q(0,T)$ for $1 \leq q < \infty$. Then there exists a subsequence $(f_{j_k})_{k \in N}$ and a function $f \in \F_\ell$ such that
\bq\label{con_f}
f_{j_k} \rightharpoonup f \quad \mbox{weakly* in } L^q(0,T;W^{1,p}(\R^d))\quad \mbox{as} \quad k \to \infty,
\eq
i.e.,
\[
\lim_{k \to \infty}\int_0^T \int_{\R^d} \phi(x,t)(f_{j_k}(x,t) - f(x,t))\,dxdt = 0 \quad \mbox{for all} \quad \phi \in L^{q'}(0,T;W^{-1,p'}(\R^d)).
\]
\end{lem}
Define the empirical measure $\mu^N$ associated to the particle system \eqref{eq_finite} as
\[
\mu^N_t := \frac1N \sum_{i=1}^N \delta_{x_i(t)} \quad \mbox{for} \quad t \geq 0.
\]
Then we are now in a position to state our theorem on the mean field limit of the optimal control problem.

\begin{thm} Let $T >0$. Suppose that $P \in W^{1,\infty}(\R^{2d})$ and $\CLambda$ satisfies that there exist $C \geq 0$ and $1 \leq q < \infty$ 
\[
Lip(\CLambda, B(0,R)) \leq CR^{q-1} \quad \mbox{for all} \quad R > 0.
\]
Let $\ell(t)$ be a fixed function in $L^q(0,T)$. Furthermore we assume that $\{x_i^0\}_{i=1}^N \subset B(0,R_0)$ for $R_0 > 0$ independent of $N$. For all $N \in \N$, let us denote the control function $f_N \in \F_\ell$ as a solution of the finite dimensional optimal control problem \eqref{cost_finite}-\eqref{eq_finite}. If there exits a compactly supported initial data $\mrho_0 \in \P_1(\R^d)$ such that $\lim_{N \to \infty}\W_1(\mu_0^N, \mrho_0)$, then there exists a subsequence $(f^{N_k}_t)_{k \in \N}$ and a function $f^\infty_t$ such that $f^{N_k}_t \to f^\infty_t$ in the sense of \eqref{con_f}. Moreover, $f^\infty_t$ and the corresponding $\mu^\infty_t$ are  solutions of the infinite dimensional optimal control problem \eqref{cost_inf}-\eqref{eq_inf}.
\end{thm}
\begin{proof} We first notice that the existence of an optimal control $f^N_t$ on the time interval $[0,T]$ for the finite dimensional optimal problem \eqref{cost_finite}-\eqref{eq_finite} can be obtained by using the weak compactness estimate in Lemma \ref{lem_comp_f} together with the strong regularity of velocity field $\mathcal P + f$, see \cite[Theorem 3.3]{FS13}. For any $f \in \F_\ell([0,T])$, let us denote $(\mu_f)^N_t$ by the solution to the equation \eqref{eq_finite} with the initial data $(\mu_f)_0^N$ satisfying $\lim_{N \to \infty} \W_1((\mu_f)_0^N,\mrho_0) = 0$. Let denote also by $\mrho^{f_t}_t$ is a solution associated to \eqref{eq_inf} with the control $f_t$ and that initial data $\mrho_0$, which is ensured by Theorem \ref{thm_ext}. Morevoer, by Theorem \ref{thm_ext}, $\lim_{N \to \infty} \W_1((\mu_f)_t^N, \mrho_t^{f_t}) = 0$. On the other hand, it follows from Lemma \ref{lem_comp_f} that there exists a subsequence $f^{N_k}_t$ such that $f^{N_k}_t \rightharpoonup f^\infty_t$ weakly* in $L^q(0,T;W^{1,p}(\R^d))$ as $k \to \infty$ for some $f^\infty_t \in \F_\ell$. Let $\mrho^{\infty}_t$ is the solution to \eqref{eq_inf} with the control function $f^\infty_t$. Then, by the lower-semicontinuity of the onset functional, we get
\[
\liminf_{k \to \infty}J\lt(\mu_t^{N_k}, f^{N_k}_t\rt) \geq J(\mrho^\infty_t,f^\infty_t),
\]
where $\mu^{N_k}_t$ is a solution to the particle equation \eqref{eq_finite} with the optimal control $f^{N_k}_t$. Then, due to the minimality of $f^{N_k}_t$, it is clear that 
\[
J\lt((\mu_f)_t^{N_k}, f_t\rt) \geq J\lt(\mu_t^{N_k}, f^{N_k}_t\rt) \quad \mbox{for each} \quad k \in \N.
\]
We finally use the convergence of $\lim_{k \to \infty} \W_1((\mu_f)_t^{N_k}, \mrho_t^f) = 0$ together with the compactly supported solution $\mrho_t$ to have
\[
J(\mrho_t^{f_t}, f_t)= \lim_{k \to \infty} J\lt((\mu_f)_t^{N_k}, f_t\rt) \geq \liminf_{k \to \infty}J\lt(\mu_t^{N_k}, f^{N_k}_t\rt) \geq J(\mrho^\infty_t,f^\infty_t).
\]
Since $f_t$ is arbitrarily chosen in $\F_\ell([0,T])$, this concludes  
\[
\min_{f_t \in \F_\ell} J(\mrho_t,f_t)= J(\mrho^\infty_t,f^\infty_t),
\]
i.e., $f^\infty_t$ is the optimal control for the problem \eqref{cost_inf}-\eqref{eq_inf}.
\end{proof}

%
%
%
%
%

\subsection{Stochastic case}
In this section, we study the parabolic optimal control problem in a bounded domain. 
In this section we are to a certain extent inspired by the work \cite{BFMW14}. As we are deviating from that in certain estimates, we take the burden somehow of presenting the results in more details than in the previous section. \\

Let $\om$ denote an open, bounded, smooth subset of $\R^d$.  We first introduce function spaces: 
\[
V:= L^2(0,T; H^1(\om)) \cap \dot{H}^1(0,T; H^{-1}_*(\om)), \quad \mbox{and} \quad H^{-1}_*(\om) = H^1(\om)',
\]
and the set of admissible controls
\[
Q_M := \lt\{ \|f\|_{L^2(0,T; L^\infty(\om))} \leq M \,:\,f \in L^2(0,T; L^\infty(\om))\rt\},
\]
for a given $M>0$.
Then our optimization problem is to show the existence of
\bq\label{eqn_mini}
\min_{f \in Q_M} J(\mrho,f) := \min_{ f \in Q_M}\int_0^T\left(\frac{1}{2}\int_\om |x-x_d|^2\mrho(x,t)\,dx + \gamma\int_\om \CLambda(f)\mrho(x,t)\,dx\right)\,dt,
\eq
where $\mrho$ is a weak solution to the following parabolic equation:
\bq\label{eqn_diff}
\pa_t \mrho  + \nabla \cdot (\QF[\mrho]\mrho + f\mrho) = \sigma\Delta \mrho, \quad (x,t) \in \om_T:=\om \times [0,T],
\eq
with the initial data
\[
\mrho(\cdot, 0) = \mrho_0(x) \quad x \in \om,
\]
and the zero-flux boundary condition
\[
\lt\lal \sigma\nabla \mrho - (\QF[\mrho] + f)\mrho, \,n(x) \rt\ral = 0, \quad (x,t) \in \pa \om \times [0,T],
\]
where $n(x)$ is the outward normal to $\pa \om$ at the point $x \in \pa \om$. Here the interaction term is given by
\[
\QF[\mrho](x,t) = \int_\om P(x,y)(y-x)\mrho(y,t)\,dy.
\]
We next provide a notion of weak solution to the equation \eqref{eqn_diff}.
\begin{defn}\label{def_weak} For a given $T >0$, a function $\mrho : \om_T \to [0,\infty)$ is a weak solution of the equation \eqref{eqn_diff} on the time-interval $[0,T]$ if and only if 
\begin{enumerate}
\item $\mrho \in L^2(0,T;H^1(\om))$ and $\pa_t \mrho \in L^2(0,T; H^{-1}_*(\om))$.
\item For any $\varphi \in L^2(0,T;H^1(\om))$, 
\[
\int_0^T \lal\pa_t \mu, \varphi \ral_{H^{-1}_* \times H^1}\,dt
- \int_0^T \int_\om \lt( \QF[\mrho] \mrho + f \mrho - \sigma\nabla \mrho \rt)\cdot \nabla \varphi\,dx dt = 0
\]
\end{enumerate}
\end{defn}

\begin{thm}\label{thm_weak} For a given $T, M >0$, let $f \in Q_M$ and $\mrho_0 \in L^2(\om)$. Furthermore, we assume $P \in L^\infty(\om^2)$. Then there exists a unique weak solution $\mrho$ to the equation \eqref{eqn_diff} in the sense of Definition \ref{def_weak}.
\end{thm}
\begin{proof} {\it Existence.-} We first employ the following iteration scheme: Let $\mrho^1(x,t) := \mrho_0(x)$ for $(x,t) \in \om_T$. For $n \geq 1$, let $\mrho^{n+1}$ be the solution of 
\[
\pa_t \mrho^{n+1} + \nabla \cdot (\QF[\mrho^n] \mrho^{n+1} + f \mrho^{n+1}) = \sigma\Delta \mrho^{n+1}
\]
with the initial data $\mrho^n(x)|_{t=0} = \mrho_0(x)$ for all $n \geq  1$ $x \in \om$ and the zero-flux boundary conditions. It is clear that $\int_\om \mrho^n(x,t)\,dx = \int_\om \mrho_0(x)\,dx$. Note that for given $\mrho^n \in V$ we can have a unique weak solution to the equation \eqref{eqn_diff} since $\QF[\mrho^n] \in L^\infty(\om)$ and $f \in L^\infty(\om)$. We next show that $\mrho^{n+1} \in V$. A straightforward computation yields
$$\begin{aligned}
\frac12\frac{d}{dt}\int_\om (\mrho^{n+1})^2\,dx + \sigma\int_\om |\nabla \mrho^{n+1}|^2\,dx &= \int_\om \nabla \mrho^{n+1} \cdot\lt( \QF[\mrho^n]\mrho^{n+1} + f\mrho^{n+1} \rt)\,dx\cr
&=: I_1 + I_2,
\end{aligned}$$
where $I_2$ can be easily estimated as
\[
I_2 \leq \int_\om |\nabla \mrho^{n+1}| |f| \mrho^{n+1}\,dx \leq \frac{\epsilon}{2}\int_\om |\nabla \mrho^{n+1}|^2\,dx + C_\epsilon \|f\|_{L^\infty}^2 \int_\om (\mrho^{n+1})^2\,dx.
\]
For the estimate of $I_1$, we use the fact that
\bq\label{est_weak1}
\|\QF[\mrho^n]\|_{L^\infty} \leq \operatorname{diam}(\om)\|P\|_{L^\infty}\|\mrho_0\|_{L^1} < \infty,
\eq
to obtain
\[
|I_1| \leq \int_\om |\nabla \mrho^{n+1}| |\QF[\mrho^n]|\mrho^{n+1}\,dx \leq \frac{\epsilon}{2}\int_\om |\nabla \mrho^{n+1}|^2\,dx + C_\epsilon\int_\om (\mrho^{n+1})^2\,dx.
\]
Combining the above estimates and choosing $\epsilon < \sigma$, we find
\[
\frac12\frac{d}{dt}\int_\om (\mrho^{n+1})^2\,dx + \lt(\sigma - \epsilon\rt)\int_\om |\nabla \mrho^{n+1}|^2\,dx \leq C_\epsilon\lt(1 + \|f\|_{L^\infty}^2 \rt)\int_\om (\mrho^{n+1})^2\,dx.
\]
Applying Gronwall's inequality to the above differential inequality deduces
\bq\label{est_weak1.5}
\int_\om (\mrho^{n+1})^2\,dx + \int_0^t \int_\om |\nabla \mrho^{n+1}|^2\,dxds \leq C(T,\sigma,\|\mrho_0\|_{L^2}, M).
\eq
We also get that for all $\psi \in H^1(\om)$
$$\begin{aligned}
\|\pa_t \mrho^{n+1}\|_{H^{-1}_*} &= \sup_{\|\psi\|_{H^1} \leq 1} | \lal \pa_t \mrho^{n+1}, \psi \ral|\cr
&\leq \sup_{\|\psi\|_{H^1} \leq 1} \lt| \lt\lal \QF[\mrho^n]\mrho^{n+1} + f \mrho^{n+1} + \sigma\nabla \mrho^{n+1}, \nabla \psi \rt\ral\rt|\cr
&\leq \lt(\|\QF[\mrho^n]\|_{L^\infty} + \|f\|_{L^\infty}\rt)\|\mrho^{n+1}\|_{L^2} + \sigma\|\nabla \mrho^{n+1}\|_{L^2}.
\end{aligned}$$
Thus we obtain $\pa_t \mrho^{n+1}\in L^2(0,T; H^{-1}_*(\om))$ due to \eqref{est_weak1} and \eqref{est_weak1.5}. This concludes $\mrho^n \in V$ for all $n \geq 2$. Note that this also implies $\mrho^{n} \in \mc([0,T];L^2(\om))$ for all $n \geq 2$. Indeed, we have
\[
\max_{0 \leq t \leq T}\|\mrho^n(t)\|_{L^2} \leq C\lt(\|\mrho^n\|_{L^2(0,T;H^1)} + \|\pa_t \mrho^n\|_{L^2(0,T;H^{-1}_*)} \rt) \quad \mbox{for all } n \geq 2,
\]
where $C$ only depends on $T$. Then, by Aubin-Lions lemma, there exist a subsequence $\mrho^{n_k}$ and a function $\mrho \in L^2(\om_T)$ such that
\bq\label{est_weak2}
\mrho^{n_k} \to \mrho \quad \mbox{in } L^2(\om_T) \quad \mbox{as} \quad k \to \infty.
\eq
We next show that the above limiting function $\mrho$ solves the equation \eqref{eqn_diff} in the sense of Definition \ref{def_weak}. For this, it suffices to take into account the interaction term $\QF[\mrho]\mrho$ since the other terms are linear with respect to $\mrho$.
Using the linearity of the functional $\QF$ together with \eqref{est_weak1} and the following fact
\[
\|\QF[f]\|_{L^\infty} \leq \operatorname{diam}(\om)\|P\|_{L^\infty} \sqrt{|\om|}\|f\|_{L^2},
\]
we get
\begin{align}\label{est_weak3}
\begin{aligned}
&\int_0^T \int_\om \lt| \mrho^{n_{k+1}}\QF[\mrho^{n_k}] - \mrho \QF[\mrho]\rt|^2 dxdt \cr
&\quad \leq  2\int_0^T \int_\om \lt| \mrho^{n_{k+1}} - \mrho\rt|^2|\QF[\mrho^{n_k}]|^2\,dxdt + 2\int_0^T \int_\om \mrho^2 |\QF[\mrho^{n_k} - \mrho]|^2\,dxdt\cr
&\quad \leq C_0\int_0^T \int_\om \lt| \mrho^{n_{k+1}} - \mrho\rt|^2 + \lt| \mrho^{n_{k}} - \mrho\rt|^2\,dxdt \to 0 \quad \mbox{as} \quad k \to \infty,
\end{aligned}
\end{align}
where $C_0 > 0$ is given by
\[
C_0 := 2 \operatorname{diam}(\om)^2\|P\|_{L^\infty}^2\lt(\|\mrho_0\|_{L^1}^2  + |\om|\|\mrho\|_{L^\infty(0,T;L^2)}^2\rt).
\]
Furthermore, we can easily show that
\[
\lim_{t \to 0+}\int_\om \mrho^{n_k +1}(x,t)\varphi(x,t)\,dx = \int_\om \mrho_0(x)\varphi_0(x)\,dx
\]
and
\[
\lim_{t \to T-}\lim_{k \to \infty}\int_\om \mrho^{n_k + 1}(x,t)\varphi(x,t)\,dx = \int_\om \mrho(x,T)\varphi(x,T)\,dx,
\]
due to $\mrho^n \in \mc([0,T];L^2(\om))$ and \eqref{est_weak2}. Hence we have that the limiting function $\mrho$ satisfies
\[
\int_{\om} \mrho(x,T)\varphi(x,T)\,dx - \int_{\om} \mrho_0(x)\varphi_0(x)\,dx=\int_0^T \int_\om \mrho \pa_t \varphi+\lt( \QF[\mrho] \mrho + f \mrho - \sigma\nabla \mrho \rt)\cdot \nabla \varphi\,dx dt .
\]

{\it Uniqueness.-} Let $\mrho_i,i=1,2$ be two solutions to the equation \eqref{eqn_diff} with initial data $\mrho_i(0) \in L^2(\om)$. Then, by using the similar estimate as in \eqref{est_weak3}, we find
$$\begin{aligned}
&\frac12\frac{d}{dt}\int_\om |\mrho_1 - \mrho_2|^2\,dx + \sigma\int_\om |\nabla(\mrho_1 - \mrho_2)|^2\,dx\cr
&\quad =  \int_\om \nabla(\mrho_1 - \mrho_2) \cdot \lt(  \QF[\mrho_1 - \mrho_2]\mrho_1 + \QF[\mrho_2](\mrho_1 - \mrho_2) + f(\mrho_1 - \mrho_2) \rt)dx\cr
&\quad \leq \e\int_\om |\nabla(\mrho_1 - \mrho_2)|^2\,dx + C_\e\lt(1 + \|f\|_{L^\infty}^2\rt)\int_\om |\mrho_1 - \mrho_2|^2\,dx,
\end{aligned}$$
where $C_\e$ depends only on $\om$, $\e$, $\|\mrho_1\|_{L^\infty(0,T;L^2)}$, and $\|\mrho_2(0)\|_{L^1}$. Finally, we apply the Gronwall's inequality to the above differential inequality to get
\[
\|\mrho_1 - \mrho_2\|_{L^\infty(0,T;L^2)}^2 + \| \nabla(\mrho_1 - \mrho_2)\|_{L^2(0,T;L^2)}^2 \leq C_1 \|\mrho_1(0) - \mrho_2(0)\|_{L^2}^2
\]
where $C_1$ depends only on $T,\sigma,\|\mrho_2(0)\|_{L^2}, M, \om$, and $ \|\mrho_1\|_{L^\infty(0,T;L^2)}$.
This completes the proof.
\end{proof}

\begin{thm} For a given $T, M> 0$, let us assume $\mrho_0 \in L^2(\om)$.  Furthermore, we assume that $P \in L^\infty(\Omega^2)$ and $\CLambda$ satisfies that for all $R > 0$
\[
W^{1,\infty}(\CLambda, B(0,R)) \leq CR,
\]
for some $C > 0$. Then there exist $f^\infty \in Q_M$ and the corresponding density $\mrho^\infty$ solving the optimal control problem \eqref{eqn_mini}-\eqref{eqn_diff}.
\end{thm}
\begin{proof} For $f \in Q_M$, by Theorem \ref{thm_weak}, there exists a weak solution $\mrho$ in the sense of Definition \ref{def_weak}. Note that $0 \in Q_M$ and 
\[
J(\mrho^0,0) = \frac12\int_0^T\int_\om |x-x_d|^2\mrho(x,t)\,dxdt \leq C(T,\om)\|\mrho_0\|_{L^1(\om)} \leq C,
\] 
where $\mrho^0$ is a weak solution of the equation \eqref{eqn_diff} with $f=0$.
Since $J(\mrho,f) \geq 0$ for all $(\mrho,f)\in V \times Q_M$, there exist a sequence $(f^j)_{j \in \N} \in Q_M$ and the corresponding density $(\mrho^j)_{j \in \N} \in V$ solving \eqref{eqn_diff} such that
\[
\lim_{j \to \infty} J(\mrho^j, f^j) = \inf_{f \in Q_M} J(\mrho,f).
\]
On the other hand, since $(\mrho^j,f^j)_{j \in \N} \in V \times Q_M$, by Banach-Alaoglu theorem,  there exist a subsequence $(\mrho^{j_k},f^{j_k}) \in V \times Q_M$ and $(\mrho^\infty,f^\infty)\in V \times Q_M$ such that
\bq\label{res_conv}
\mrho^{j_k} \to \mrho^\infty \quad \mbox{in } L^2(\om_T) \quad \mbox{and} \quad f^{j_k} \weakto f^\infty \quad \mbox{in } L^2(0,T;L^\infty(\om)).
\eq
We next show that $(\mrho^\infty,f^\infty)$ is a solution to \eqref{eqn_diff}. For this, it is enough to show that
\[
I_k:= \int_0^T \int_\om \lt(f^{j_k}\mrho^{j_k} - f^\infty \mrho^\infty\rt)\phi\,dxdt \to 0\quad \mbox{as} \quad k \to  \infty,
\]
for $\phi \in L^2(0,T;H^1(\om))$. For this, we decompose $I_k$ into two parts as
\[
I_k = \int_0^T \int_\om (f^{j_k} - f^\infty)\mrho^{j_k} \phi\,dxdt + \int_0^T \int_\om (\mrho^{j_k} - \mrho^\infty)f^\infty \phi\,dxdt =: I_k^1 + I_k^2.
\]
Since
\[
L^2(0,T;L^\infty(\om)) = \lt(L^2(0,T;L^1(\om))\rt)' \quad \mbox{and} \quad \mrho^{j_k} \phi \in L^2(0,T;L^1(\om)),
\]
it is clear from \eqref{res_conv} that  $I_k^1 \to 0$ as $k \to \infty$. For the convergence of $I_k^2$, we get
$$\begin{aligned}
I_k^2 &\leq \int_0^T\|f^\infty\|_{L^\infty}\|\mrho^{j_k} - \mrho^\infty\|_{L^2}\|\phi\|_{L^2}\,dt \cr
&\leq \|\phi\|_{L^\infty(0,T;L^2)}\|f^\infty\|_{L^2(0,T;L^\infty)}\|\mrho^{j_k} - \mrho^\infty\|_{L^2(0,T;L^2)} \to 0 \quad \mbox{as} \quad k \to \infty.
\end{aligned}$$
Thus we conclude that $(\mrho^\infty,f^\infty)$ is a solution to \eqref{eqn_diff}. Furthermore, we obtain
\[
\int_0^T \int_\om |x - x_d|^2 \mrho^{j_k}\,dxdt \to \int_0^T \int_\om |x - x_d|^2 \mrho^\infty\,dxdt \quad \mbox{as} \quad k \to \infty,
\]
due to $|\om| < \infty$. We also find
\bq\label{est_semi}
\lim_{k \to \infty}\int_0^T \int_\om \CLambda(f^{j_k}) \mrho^{j_k}\,dxdt \geq \int_0^T \int_\om \CLambda(f^\infty) \mrho^\infty\,dxdt.
\eq
More precisely, we can show that
\[
\CLambda(f^{j_k}) \mrho^{j_k} \weakto \CLambda(f^\infty) \mrho^\infty \quad \mbox{in } \mathcal{M}(\om_T) \quad \mbox{as} \quad k \to \infty.
\]
Indeed, for $\phi \in \mc_c(\om_T)$, we have
$$\begin{aligned}
&\int_0^T \int_\om \lt(\CLambda(f^{j_k}) \mrho^{j_k} - \CLambda(f^\infty) \mrho^\infty \rt)\phi\,dxdt \cr
&\quad = \int_0^T \int_\om \lt(\CLambda(f^{j_k}) - \CLambda(f^\infty)\rt)\mrho^{j_k}\phi\,dxdt +\int_0^T \int_\om \CLambda(f^\infty)(\mrho^{j_k} - \mrho^\infty)\phi\,dxdt \cr
&\quad =: J_k^1 + J_k^2,
\end{aligned}$$
where $J_k^2$ is easily estimated by
\[
J_k^2 \leq C\|\phi\|_{L^\infty(\om_T)}\int_0^T \|f^\infty\|_{L^\infty}\|\mrho^\infty - \mrho^\infty\|_{L^2}\,dt \leq M\|\phi\|_{L^\infty(\om_T)}\|\mrho^{j_k} - \mrho^\infty\|_{L^2(\om_T)}.
\]
Thus $J_k^2 \to 0 $ as $k \to \infty$. For the estimate of $J_k^1$, we note that there exists a $c_0 \in [0,1]$ such that
\[
\CLambda(f^{j_k}) - \CLambda(f^\infty) = \nabla \CLambda \lt( (1 - c_0)f^\infty - c_0 f^{j_k} \rt)\cdot (f^{j_k} - f^\infty).
\]
On the other hand, it follows the assumption on $\CLambda$ that
\[
\int_{\R^d} \lt|\nabla \CLambda \lt( (1 - c_0)f^\infty - c_0 f^{j_k} \rt) \rt||\phi| \mrho^{j_k}\,dx \leq C\|\phi\|_{L^\infty}\lt(\|f^{j_k}\|_{L^\infty} + \|f^\infty\|_{L^\infty}\rt),
\]
and this implies 
\[
\nabla \CLambda \lt( (1 - c_0)f^\infty - c_0 f^{j_k} \rt) \mrho^{j_k}\phi \in L^2(0,T;L^1(\om)) \quad \mbox{uniformly in } k.
\]
This yields $J_k^1 \to 0$ as $k \to \infty$. Then, by de la Vall\'ee-Poussin's theorem, we obtain the semicontinuity \eqref{est_semi}. This yields
\[
\liminf_{k \to \infty}J(\mrho^{j_k}, f^{j_k}) \geq J(\mrho^\infty,f^\infty).
\]
Hence we conclude 
\[
\inf_{f \in Q_M} J(\mrho,f)=\lim_{j \to \infty} J(\mrho^j, f^j)=\liminf_{k \to \infty}J(\mrho^{j_k}, f^{j_k}) \geq J(\mrho^\infty,f^\infty).
\]

\end{proof}

\section{First order optimality conditions}\label{sec:firstorderopt}
In this section, we derive first order optimality conditions for the mean field optimal control problem studied in Section \ref{sec:existence}:
\begin{align}\label{eq:MFOC}
\partial_t\mu + \nabla \cdot \lt((\QF[\mu] + f)\mu\rt) = \sigma \Delta \mu, \quad x \in \om, \quad t > 0,
\end{align}
where the control $f$ is the solution of the minimization of the following cost functional:
\begin{align}\label{eq:Jfun}
J(\mu,f) = \int_0^T\left(\frac{1}{2}\int_\om |x-x_d|^2\mu(x,t)\,dx + \gamma \int_\om \CLambda(f)\mu(x,t)\,dx\right)\,dt.
\end{align}
\subsection{Formal derivation of the optimality conditions}
Let us first write the Lagrangian of the mean field optimal control defined by \eqref{eq:MFOC} and \eqref{eq:Jfun}, as follows
\begin{equation}
\begin{aligned}\label{eq:Lagrangian}
&\mathcal{L}(\mu,\psi,f) = \int_0^T\left(\frac12\int_\om |x-x_d|^2\mu(x,t)\,dx +\gamma \int_\om \CLambda(f) \mu(x,t)\,dx\right)dt \\
& \qquad -  \int_0^T\left[\int_\om \psi(x,t)\left(\partial_t\mu(x,t) + \nabla \cdot \lt(\QF[\mu(x,t)] + f(x,t))\mu(x,t)\rt) - \sigma \Delta \mu(x,t)\right)\,dx\right]dt.
\end{aligned}
\end{equation}
Integrating by parts and taking the terminal data $\psi(x,T) = 0$, we get
\begin{equation}
\begin{aligned}\label{eq:Lagrangian2}
\mathcal{L}(\mu,\psi,f) = &\int_0^T\left(\frac12\int_\om |x-x_d|^2\mu \,dx +\gamma \int_\om \CLambda(f) \mu\,dx\right)\,dt  +\int_\om \psi(x,0)\mu(x,0)\,dx + \int_0^T\int_\om \partial_t\psi \,\mu \,dx dt\\
& + \int_0^T\int_\om \nabla \psi  \cdot (\bigF[\mu] \mu)\,dxdt  + \int_0^T\int_\om \nabla \psi \cdot (f \mu)\,dxdt + \sigma \int_0^T \int_\om \mu\Delta \psi\,dxdt,
\end{aligned}
\end{equation}
where we omit the dependency on $(x,t)$ where not necessary. We compute the functional derivatives of the Lagrangian with respect to the state function $\mu$ and the control $f$,
\begin{align}
\frac{\delta \mathcal{L}}{\delta f} &= \gamma \nabla \CLambda (f) \mu - \nabla \psi\,\mu = (\gamma \nabla \CLambda(f)-\nabla \psi)\mu,\label{eq:OptCon}\\
\frac{\delta \mathcal{L}}{\delta \mu} &=\frac{1}{2}|x-x_d|^2 +\gamma \CLambda(f) + \partial_t\psi + \nabla \psi \cdot f + \sigma \Delta \psi\cr
&\quad - \frac12\int_\om \lt(P(x,y)\nabla \psi(x,t) - P(y,x)\nabla \psi(y,t)\rt)\cdot (y-x)\mu(y,t)\,dy.\label{eq:Adj}
\end{align}
Let $(\mu^*,\psi^*,f^*)$ be the solution to the optimal control problem. Then we have
\[
\frac{\delta \mathcal{L}}{\delta f}\Big|_{(\mu,\psi,f) = (\mu^*,\psi^*,f^*)} = 0 \quad \mbox{and} \quad \frac{\delta \mathcal{L}}{\delta \mu}\Big|_{(\mu,\psi,f) = (\mu^*,\psi^*,f^*)} = 0.
\]
This yields from \eqref{eq:OptCon} that
\bq\label{condi_fs}
\gamma \nabla \CLambda (f^*) = \nabla \psi^* \quad \mbox{on the support of $\mu^*$}.
\eq
We also find from \eqref{eq:Adj} that $\psi^*$ satisfies 
$$\begin{aligned}
&\pa_t \psi^* + \frac12|x-x_d|^2 + \gamma \CLambda(f^*) + \nabla \psi^*  \cdot f^* + \sigma \Delta  \psi^* \cr
&\qquad \quad - \frac12\int_\om \lt(P(x,y)\nabla \psi^*(x,t) - P(y,x)\nabla \psi^*(y,t)\rt)\cdot (y-x)\mu^*(y,t)\,dy = 0,
\end{aligned}$$
or equivalently
\begin{align}\label{eq_fpsi}
\begin{aligned}
&\pa_t \psi^* + \frac12|x-x_d|^2 + \gamma\lt( \CLambda(f^*) + \nabla \CLambda(f^*) \cdot f^*\rt) + \sigma \Delta \psi^* \cr
&\qquad \quad - \frac12\int_\om \lt(P(x,y)\nabla \psi^*(x,t) - P(y,x)\nabla \psi^*(y,t)\rt)\cdot (y-x)\mu^*(y,t)\,dy = 0,\end{aligned}
\end{align}
due to \eqref{condi_fs}, where $\mu^*$ satisfies
\[
\partial_t\mu^* + \nabla \cdot ((\bigF[\mu^*] + f^*)\mu^*) = \sigma\Delta \mu^*\quad \mbox{with} \quad \nabla \CLambda (f^*) = \frac1\gamma \nabla \psi^*.
\]
\subsection{Rigorous derivation of the optimality conditions}

The first order optimality conditions \eqref{firstordercond} are of utmost relevance as they are often used for the numerical computation of mean field optimal controls and we show how to proceed for that in Section \ref{sec:num}. Although they are very often {\it formally} derived, as we do above, and used  in several contributions, see, e.g. \cite{befrph13}, as a relatively straightforward consequence of the Lagrange multiplier theorem, we feel that presenting their {\it rigorous} derivation can be useful for a reader not familiar with such derivations. Moreover, by doing so, we highlight more precisely certain technical difficulties and aspects, which one may in fact encounter along the process, and are often left to a certain extent as for granted. Let us recall then the Lagrange multiplier theorem in Banach spaces.\\

Let $X$ and $Y$ be Banach spaces, and let a functional $J: U(x^*) \subseteq X \to \R$ and a mapping $G: U(x^*)\subseteq X \to Y$ be continuously differential on an open neighbourhood of $x^*$. Consider the following optimal problem:
\bq\label{eq_op}
J(x) \to \inf, \quad G(x) = 0.
\eq
Then we recall the following first order optimality condition whose proof can be found in \cite[Section 4.14]{Zei}.
\begin{thm}\label{thm_oc}Let $x^*$ be a solution to the problem \eqref{eq_op}, and let the range of the operator $G'(x^*) : X \to Y$ be closed. Then there exists a nonzero pair $(\lambda,p) \in \R \times Y'$ such that
\[
\mathcal{L}'_x(x^*,\lambda,p)(x) = 0 \quad \mbox{for all } x \in X,
\]
where 
\[
\mathcal{L}(x,\lambda,p) = \lambda J(x) + G(x)(p).
\]
Moreover, if Im $G'(x^*) = Y$, then $\lambda \neq 0$ in the above, thus we can assume that $\lambda = 1$.
\end{thm}

In order to apply the above theorem, we set
\[
X = V \times L^2(\om_T), \quad Y = L^2(0,T;H^{-1}(\om)),  
\]
\[
J(\mrho,f) = \int_0^T\left(\frac{1}{2}\int_\om |x-x_d|^2\mrho(x,t)\,dx + \gamma\int_\om \CLambda(f)\mrho(x,t)\,dx\right)dt,
\]
and
$$\begin{aligned}
G(\mu,f)(\psi) &= - \int_\om \psi(x,T)\mu(x,T)\,dx +\int_\om \psi(x,0)\mu(x,0)\,dx + \int_0^T\int_\om \partial_t\psi \,\mu \,dx dt\\
& \quad + \int_0^T\int_\om \nabla \psi  \cdot (\bigF[\mu] \mu)\,dxdt  + \int_0^T\int_\om \nabla \psi \cdot (f \mu)\,dxdt - \sigma \int_0^T \int_\om\nabla \mu \cdot \nabla \psi\,dxdt,
\end{aligned}$$
for $\psi \in Y' = L^2(0,T;H^1_0(\om))$. Then straightforward computations yield
$$\begin{aligned}
G'_\mu(\mu,f)(\nu,\psi) &=- \int_\om \psi(x,T)\nu(x,T)\,dx +\int_\om \psi(x,0)\nu(x,0)\,dx + \int_0^T\int_\om \partial_t\psi \,\nu \,dx dt\\
& \quad + \int_0^T\int_\om \nabla \psi  \cdot (\bigF[\nu] \mu + \bigF[\mu] \nu + f\nu)\,dxdt  - \sigma \int_0^T \int_\om \nabla \nu\cdot \nabla \psi\,dxdt, 
\end{aligned}$$
for $(\nu,\psi) \in V \times Y'$, and 
\[
G'_f(\mu,f)(g,\psi) = \int_0^T \int_\om \nabla \psi \cdot (g\mu)\,dxdt \quad \mbox{for} \quad (g,\psi) \in Q_M \times V'.
\]
Note that the interaction terms on the right hand side of the equality for $G'_\mu(\mu,f)(\nu,\psi)$ can be rewritten as
$$\begin{aligned}
&\int_0^T\int_\om \nabla \psi  \cdot (\bigF[\nu] \mu + \bigF[\mu] \nu) dxdt \cr
&\quad = \frac12\int_0^T \int_{\om^2} \lt(P(x,y)\nabla \psi(x) - P(y,x)\nabla\psi(y) \rt)\cdot(y-x) \lt(\nu(x)\mu(y) + \mu(x)\nu(y) \rt)dxdydt.
\end{aligned}$$

We now present our main result on the first order optimality condition in the theorem below.
\begin{thm} Let $(\mu^*,f^*) \in V \times Q_M$ be a solution to the problem  \eqref{eq:MFOC}-\eqref{eq:Jfun}. Suppose that there exists a $\mu_\ell > 0$ such that $\mu^* \geq \mu_\ell$ for all $(x,t) \in \om_T$. Then there exists $\psi^* \in Y'$ such that
\begin{equation}\label{firstordercond} \begin{aligned}
G'_\mu(\mu^*,f^*)(\nu, \psi^* ) &= J'_\mu(\mu^*,f^*)(\nu), \quad \mbox{for all} \quad \nu \in V,\\
G'_f(\mu^*,f^*)(g, \psi^*) &= J'_f(\mu^*,f^*)(g), \quad \mbox{for all} \quad g \in L^2(\om_T).
\end{aligned}\end{equation}
\end{thm}
Before presenting the proof of the first order optimality conditions \eqref{firstordercond}, let us comment the positivity principle on the existence of $\mu_\ell > 0$ such that $\mu^* \geq \mu_\ell$ for all $(x,t) \in \om_T$. We can verify this property numerically, as shown in Section \ref{sec:num}, but we expect it to follow from an appropriate maximum principle, currently still under investigation. For now we consider this technical condition as acceptable.

\begin{proof} For the proof, we show that linear operators $G'_\mu (\mu^*,f^*): V \to Y$ and $G'_f (\mu^*,f^*): L^2(\om_T)\lt(\supseteq Q_M\rt) \to Y$ are surjective. Then, by Theorem \ref{thm_oc}, we conclude our desired results.

{\bf Surjectivity of $G'_\mu(\mu^*,f^*)$.-} Let $(\mu^*,f^*) \in V \times Q_M$ be a solution to \eqref{eq:MFOC}-\eqref{eq:Jfun}. We want to show that for any $\eta \in Y$ there exists a $\nu \in V$ such that 
\[
G'_\mu(\mu^*,f^*)(\nu) = \eta, \quad \mbox{i.e.,} \quad G'_\mu(\mu^*,f^*)(\nu, \psi) = \eta(\psi) \quad \mbox{for all} \quad \psi \in Y'.
\]
Note that finding the above equality is equivalent to show that for given $(\mu^*,f^*,\eta) \in V \times Q_M \times Y$, there exists a solution $\nu \in V$ to the Cauchy problem:
\bq\label{eq_nu}
\pa_t \nu + \nabla \cdot \lt(\bigF[\nu]\mu^* + \bigF[\mu^*]\nu + f^*\nu \rt) = \sigma \Delta\nu -\eta, \quad x \in \om, \quad t > 0,
\eq
with the initial data $\nu_0 \in L^2(\om)$ and the boundary condition:
\[
\lt\lal \sigma\nabla \nu - \bigF[\nu]\mu^* - \lt( \bigF[\mu^*]+ f^*\rt)\nu, n(x)\rt\ral = 0, \quad (x,t) \in \pa \om \times \R_+.
\]
We notice that \eqref{eq_nu} is linear parabolic equation of $\nu$. Thus the existence of $\nu \in V$ is enough to show the following a priori estimates which are very similar to that in the proof of Theorem \ref{thm_weak}:
$$\begin{aligned}
&\frac12\frac{d}{dt}\|\nu\|_{L^2}^2 + \sigma\|\nabla \nu\|_{L^2}^2 \leq \|\nabla \nu\|_{L^2}\lt(\|\bigF[\nu]\mu^*\|_{L^2} + \|\bigF[\mu^*]\nu\|_{L^2} + \|f^* \nu\|_{L^2}\rt) + \|\eta\|_{H^{-1}}\|\nu\|_{H^1}\cr
&\quad \leq \frac\sigma2\|\nabla \nu\|_{L^2}^2 + C\lt(\|\bigF[\nu]\|_{L^\infty}^2\|\mu^*\|_{L^2}^2 + \lt(\|\bigF[\mu^*]\|_{L^\infty}^2 + \|f^*\|_{L^\infty}^2\rt)\|\nu\|_{L^2}^2\rt)+ \|\eta\|_{H^{-1}}^2 + \|\nu\|_{L^2}^2\cr
&\quad \leq \frac\sigma2\|\nabla \nu\|_{L^2}^2 + C\lt(\|\mu^*\|_{L^2}^2 + \|f^*\|_{L^\infty}^2 + 1\rt)\|\nu\|_{L^2}^2+  \|\eta\|_{H^{-1}}^2,\cr
&\|\pa_t \nu\|_{H^{-1}} \leq \|\bigF[\nu]\|_{L^\infty}\|\mu^*\|_{L^2} + \lt(\|\bigF[\mu^*]\|_{L^\infty} + \|f^*\|_{L^\infty}\rt)\|\nu\|_{L^2} + \sigma\|\nabla \nu\|_{L^2} + \|\eta\|_{H^{-1}}\cr
&\quad \ls \lt(\|\mu^*\|_{L^2} + \|f^*\|_{L^\infty}\rt)\|\nu\|_{L^2} + \sigma\|\nabla \nu\|_{L^2} + \|\eta\|_{H^{-1}}.
\end{aligned}$$ 
Here we used 
\[
\|\bigF[\nu]\|_{L^\infty} \leq \operatorname{diam}(\om)\sqrt{|\om|}\|P\|_{L^\infty}\|\nu\|_{L^2},
\]
and similarly
\[
\|\bigF[\mu^*]\|_{L^\infty} \leq \operatorname{diam}(\om)\sqrt{|\om|}\|P\|_{L^\infty}\|\mu^*\|_{L^2}.
\]
This yields
$$\begin{aligned}
&\|\nu(\cdot,t)\|_{L^2}^2 + \int_0^t \|\nabla \nu(\cdot,s)\|_{L^2}^2 ds \cr
&\quad \leq \lt( \|\nu_0\|_{L^2}^2 + \|\eta\|_{L^2(0,T;H^{-1})}^2 \rt)\exp\lt(C\int_0^T \lt(\|\mu^*(\cdot,s)\|_{L^2}^2 + \|f^*(\cdot,s)\|_{L^\infty}^2 + 1\rt)ds \rt)
\end{aligned}$$
and
$$\begin{aligned}
\|\pa_t \nu\|_{L^2(0,T;H^{-1})} &\ls \|\nu\|_{L^\infty(0,T;L^2)}\lt(\|\mu^*\|_{L^2(\om_T)} + \|f^*\|_{L^2(0,T;L^\infty)} \rt)  \cr
&\quad + \sigma\|\nabla \nu\|_{L^2(\om_T)} + \|\eta\|_{L^2(0,T;H^{-1})}.
\end{aligned}$$

{\bf Surjectivity of $G'_f(\mu^*,f^*)$.-} For $\xi \in Y$, we first consider the following weak formulation of Poisson equation:
\bq\label{eq_poi}
\int_0^t \int_\om \nabla \psi \cdot \nabla u\,dxds = \int_0^t \int_\om \xi \psi\,dxds, \quad \mbox{for any } \psi \in H^1_0(\om),
\eq
where we already took account the space-time decomposition of the test function. To solve the equation \eqref{eq_poi}, we use the Galerkin method, i.e., we first construct a series of approximate solutions of the form:
\[
u_k(x,t) = \sum_{j=1}^k \hat u_{k,i}(t)\psi_i(x),
\]
where $(\psi_i)_{i=1}^\infty$ is an orthonormal basis for $L^2(\om)$ formed from the eigenfunctions of the Laplace operator:
\bq\label{eq_lap}
-\Delta \psi_i = \lambda_i \psi_i, \quad \psi_i \in \mc^\infty_0(\om).
\eq
It follows from the above that $\lambda_i$ can be easily computed as
\[
\lambda_i = \int_\om |\nabla \psi_i|^2\,dx > 0.
\]
Let us deal with the case $u = u_k$ in \eqref{eq_poi}. Then we obtain
\[
\sum_{i=1}^k \int_0^t \hat u_{k,i}(s) ds \int_\om \nabla \psi \cdot \nabla \psi_i\,dx  = \int_0^t \int_\om \xi \psi\,dxds.
\]
This and together with \eqref{eq_lap} yields
\[
\sum_{i=1}^k \int_0^t \hat u_{k,i}(s) ds  \int_\om \lambda_i\psi\, \psi_i\,dx  = \int_0^t \int_\om \xi \psi\,dxds.
\]
Then, by taking $\psi = \psi_i$ in the above, we get
\[
\lambda_i \hat u_{k,i}(t) = \int_\om \xi \psi_i\,dx,
\]
and by multiplying $\psi_i$ to the above and summing that over $i$, we find
\[
-\Delta u_k(x,t) = \sum_{i=1}^k \lt(\int_\om \xi(x,t) \psi_i(x)\,dx \rt) \psi_i(x),
\]
where we used \eqref{eq_lap}. This implies
$$\begin{aligned}
\int_\om |\nabla u_k(x,t)|^2\,dx &= \sum_{i=1}^k \lt(\int_\om \xi(x,t) \psi_i(x)\,dx \rt) \int_\om \psi_i(x) u_k(x,t)\,dx \cr
&=\sum_{i=1}^k \lt(\int_\om \xi(x,t) \psi_i(x)\,dx \rt)\hat u_{k,i}(t)\cr
&=\int_\om \xi(x,t) u_k(x,t)\,dx\cr
&\leq \|\xi\|_{H^{-1}}\|u_k\|_{H^1}.
\end{aligned}$$
Applying the Poincar\'e inequality to the above, we obtain
\[
\|u_k(\cdot,t)\|_{H^1} \leq C\|\xi(\cdot,t)\|_{H^{-1}},
\]
in particular, we have $u_k \in L^2(0,T;H^1(\om))$ uniformly in $k$ due to $\xi \in Y = L^2(0,T;H^{-1}(\om))$. This implies that there exist a function $u \in L^2(0,T;H^1(\om))$ such that $u_k$ converges to $u$ weakly in $L^2(0,T;H^1(\om))$ up to a subsequence. It is also easy to check that the limiting function $u$ is the solution to the equation \eqref{eq_poi}. 

We now get back to our original problem. Our goal was to show that for given $\mu^* \in V$ and $\xi \in Y$, there exists a function $g \in L^2(\om_T)$ such that 
\[
\int_0^T \int_\om \nabla \psi \cdot (g\mu^*)\,dxdt = \int_0^T \int_\om \xi\,\psi \,dxdt \quad \mbox{for any } \psi \in Y'.
\]
Then we now construct the solution $g$ to the above equation by
\[
g\mu^* = \nabla u, \quad \mbox{i.e.,} \quad g = \frac{\nabla u}{\mu^*} \quad \mbox{on the support of } \mu^*,
\]
where the existence of $u \in L^2(0,T;H^1(\om))$ was guaranteed in the beginning of the proof. Moreover, by the assumption $\mu^*(x,t)>\mu_\ell > 0$ in $\om \times [0,T]$, we have
\[
\int_0^T \int_\om |g(x,t)|^2\,dxdt = \int_0^T \int_\om \lt|\frac{\nabla u(x,t)}{\mu^(x,t)}\rt|^2\,dxdt \leq \frac{1}{\mu_\ell^2}\int_0^T \int_\om |\nabla u(x,t)|^2\,dxdt < \infty,
\]
due to $u \in L^2(0,T;H^1(\om))$. This completes the proof.

\end{proof}


\section{Hierarchy of controls via the Boltzmann equation}\label{sec:Boltz}

For large values of $N$, the solution of finite horizon control problems of the type \eqref{eq:HKu}--\eqref{eq:funcJu} through standard methods stumble upon prohibitive computational costs, due to the nonlinear constraints and the lack of convexity in the cost. Although  mean field optimal controls  \eqref{main_eq}-\eqref{main_func}  are designed to be independent of the number $N$ of agents to provide a way to circumvent the course of dimensionality of $N\to \infty$, still their numerical computation needs to be realized by solving the first-order optimality conditions. The complexity of their solution depends on the intrinsic dimensionality $d$ of the agents, which is affordable only at moderate dimensions (e.g., $d \leq 3$).  In order to tackle these difficulties, we introduce a novel reduced setting, by introducing a binary dynamics whose evolution can be described by means of a Boltzmann-type equation, \cite{APb,PTa}. Hence we will show that this description, under a proper scaling \cite{T,VILL}, converges to the mean field equation \eqref{main_eq}, \cite{AHP, CPT,T}. This type of approach allows to embed the control dynamics into two different ways:
\begin{itemize}
\item[(i)] we can assume the control $f$ to be a given function, possibly obtained from the solution of the optimal control problem \eqref{eq:HKu}--\eqref{eq:funcJu};
\item[(ii)] alternatively, the control is obtained as a solution of the reduced optimal control problem associated to the dynamics of two single agents.  We refer to this approach as {\em binary control}.
\end{itemize}

Similar ideas have been used in a control context in \cite{ABCK15,AHP,APZa,DHL14,FWb}. We devote the forthcoming sections to show different strategies to derive such binary controls. Thus we want to approach the mean field optimal control problem \eqref{eq:HKu}--\eqref{eq:funcJu} as the last step of a control hierarchy, starting from an instantaneous control strategy and going towards a binary Hamilton-Jacobi-Bellmann control.

\subsection{Binary controlled dynamics}\label{sec:bcd}
We consider the discrete controlled system \eqref{eq:HKu}--\eqref{eq:funcJu}  in the simplified case of only two interacting agents $(x_i(t),x_j(t))$ and in absence of noise, i.e. $\sigma = 0$. Hence, by defining the sample time $\Delta t$ such that $t_m = m\Delta t$, so that $0=t_0<\ldots<t_m<\ldots<t_M=T$ and introducing a forward Euler discretization, we write  \eqref{eq:HKu} as follows 
\begin{equation}
\begin{aligned}\label{eq:disc_micr2}
x_i^{m+1} = \,& x_i^m        + \frac{\Delta t}{2} P(x_i^m,x_j^m)(x_j^m-x_i^m)    +  {\Delta t} u^m_i,        
\\
x_j^{m+1}=\, & x_j^m+  \frac{\Delta t}{2} P(x_j^m,x_i^m)(x_i^m-x_j^m)   +  {\Delta t} u^m_j,
\end{aligned}
\end{equation}
where  from now on we denote the control pair $u:=(u_i,u_j)$ associated to the state variable $x:=(x_i,x_j)$, and having used the compact notation for $x^m_i=x_i(t_m), u^m_i=u_i(t_m)$.

The discretized form for the functional \eqref{eq:funcJu} for the binary dynamics \eqref{eq:disc_micr2} reads
\begin{align}\label{eq:disc_funcJu}
J_M(x,u) : = \sum_{m=0}^{M-1}\int_{t_m}^{t_{m+1}} L \left(x(t),u(t)\right)\ dt,
\end{align}
where the stage cost is given by
\begin{equation}\label{eq:disc_func}
L(x,u) =\frac{1}{2}\left(|x_i-x_d|^2+|x_j-x_d|^2\right) +\gamma \left(\CLambda(u_i)+ \CLambda(u_j)\right). 
\end{equation}
In the following we propose two alternative methods in order to characterize  $u_i,u_j$ as (sub-)optimal feedback controller. In both cases, we  will consider the controlled dynamics in the deterministic case. Nonetheless, we will show in Section \ref{sec:exp} that such controls are robust with respect to the presence of noise, ($\sigma > 0$) and they shall be employed in the corresponding stochastic setting as well.

\subsubsection{Instantaneous control}\label{sec:IC}
A first approach  towards obtaining a low complexity computational realization of the solution of the optimal control problem \eqref{eq:disc_micr2}--\eqref{eq:disc_funcJu}  is the so-called {\em model predictive control} (MPC). This strategy furnishes  a suboptimal control by an iterative solution over a sequence of finite time steps, representing the {\em predictive horizon}  \cite{AHP, APTZ, MRRS}. 
Since we are only interested in instantaneous control strategies, we  limit the MPC method to a single time prediction horizon, therefore we reduce the original optimization into the minimization on every time interval $[t_m,t_{m+1}]$ of the following functional
\begin{equation}\label{eq:disc_ist}
\begin{aligned}
J_{\Delta t}(x^m,{u}^m)&= \Delta tL(x(t_{m+1}),u(t_m))\\
					 &= \Delta t\left(\frac{1}{2}\left(|x^{m+1}_i-x_d|^2+|x^{m+1}_j-x_d|^2\right) +\gamma \left(\CLambda(u^m_i)+ \CLambda(u^m_j)\right) \right).
\end{aligned}
\end{equation}
Note that from \eqref{eq:disc_micr2} we have that $x^{m+1}$ depends linearly on $u^m$, thus $$U^m_{ij}:= U(x_i,x_j,t_m)=\underset{u^m}{\arg\min}\;J_{\Delta t}(x^m,u^m)$$ can be directly computed from the following system
 \begin{equation}
 \begin{aligned}\label{eq:IC_gen}
\Delta t ^2{U}^m_{ij} + 2\gamma\nabla_{{u}_i}\CLambda(U^m_{ij}) + \Delta t (x^m_i-x_d) +  \frac{\Delta t ^2}{2} P(x_i^m,x_j^m)(x_j^m-x_i^m) = 0,\\
\Delta t ^2{U}^m_{ji} + 2\gamma\nabla_{{u}_j}\CLambda(U^m_{ji}) + \Delta t(x^m_j-x_d) +  \frac{\Delta t ^2}{2}P(x_j^m,x_i^m)(x_i^m-x_j^m) = 0.
\end{aligned}
\end{equation}

In the case of a quadratic penalization of the control, i.e. $\CLambda(c) := |c|^2/2$, we can furnish the following explicit expression for the minimizers  
 \begin{equation}
 \begin{aligned}\label{eq:IC}
U_{ij}^m= \frac{\Delta t}{2\gamma+\Delta t^2}\left((x_d-x^m_i) -  \frac{\Delta t}{2} P(x_i^m,x_j^m)(x_j^m-x_i^m)\right) ,\\
U_{ji}^m = \frac{\Delta t}{2\gamma+\Delta t^2}\left((x_d-x^m_j) -  \frac{\Delta t}{2} P(x_j^m,x_i^m)(x_i^m-x_j^m)\right) ,\\
\end{aligned}
\end{equation}
hence \eqref{eq:IC_gen} gives a feedback control for the full binary dynamics, which can be plugged as an {\em instantaneous control} into  \eqref{eq:disc_micr2}.

\begin{remark}
Note that the instantaneous control \eqref{eq:IC} embedded into the discretized dynamics \eqref{eq:disc_micr2}, is of order $o(\Delta t)$.  To obtain an effective contribution of the control in the dynamics we will assume that the penalization parameter $\gamma$ scales with the time discretization, in this way the leading order is recovered, \cite{AHP,APTZ}, e.g. for $\gamma = \Delta t \bar\gamma$ we have
 \begin{equation}
 \begin{aligned}\label{eq:IC2}
U_{ij}^m= \frac{1}{2\bar \gamma+\Delta t}\left((x_d-x^m_i) -  \frac{\Delta t}{2} P(x_i^m,x_j^m)(x_j^m-x_i^m)\right).
\end{aligned}
\end{equation}
\end{remark}

\subsubsection{Finite horizon optimal control}\label{sec:HJB}
The instantaneous feedback control derived in the previous section is the optimal control action for the binary system with a single step prediction horizon. An improved, yet more complex optimal feedback synthesis can be performed by considering an extended finite horizon control problem. Let us define the value function associated to the finite horizon discrete cost \eqref{eq:disc_funcJu} as
\begin{align}\label{eq:Val}
V(x_i,x_j,t_m) := \underset{u\in\U}{\inf}\sum_{k=m}^{M-1}  \Delta t L(x_i(t_k),x_j(t_k),u(t_k)),\qquad \text{for } m = 0,\ldots,M-1,
\end{align}
with terminal condition $V(x_i,x_j,t_M)=0$. It is well-known that the application of the Dynamic Programming Principle \cite{BELL} with the discrete time dynamics \eqref{eq:disc_micr2} characterizes the value function as the solution of the following recursive Bellman equation 
\begin{equation}
\begin{aligned}\label{eq:DP}
V(x_i,x_j,t_M) & = 0,\\
V(x_i,x_j,t_m) & = \inf_{u\in \U}\left\{ \Delta t L(x_i,x_j,u) + V(x+\Delta t (F(x_i,x_j)+u),t_{m+1})  \right\},\ m = M-1,\ldots,0\,,
\end{aligned}
\end{equation}
where $x=(x_i,x_j)$, $u=(u_i,u_j)$,  and $F(x_i,x_j) := (P(x_i,x_j)(x_j-x_i),P(x_i,x_j)(x_j-x_i))$. 
Once this functional relation has been solved, for every time step the optimal control is recovered from the optimality condition as follows
\begin{equation}\label{eq:ochj}
U(x_i,x_j,t_m)=\underset{u\in \U}{\arg\min}\left\{ \Delta t L(x_i,x_j,u) + V(x+\Delta t (F(x_i,x_j)+u),t_{m+1})  \right\}\,.
\end{equation}
As in the expression \eqref{eq:IC_gen}, this optimal control is also in feedback form, depending not only on the current states of binary system $(x_i,x_j)$, but also on the discrete time variable $t_m$.

\begin{remark}\label{quasiopt} The system \eqref{eq:DP} is a first-order approximation of the Hamilton-Jacobi-Bellman equation
\bq\label{eq:V}
\displaystyle \partial_tV(x,t) + \inf_{u \in \U} \left\{L(x,u)+ \nabla V(x,t)\cdot\left[F(x)+u\right]\right\}=0,
\eq
related to the continuous time optimal control problem. In fact, this latter equation corresponds to the adjoint \eqref{eq:Adj} when the nonlocal integral terms are neglected, and therefore this approach although optimal for the binary system, cannot be expected to satisfy the optimality system \eqref{eq:OptCon}--\eqref{eq:Adj} related to the mean field optimal control problem.
\end{remark}

\subsection{Boltzmann description}
We introduce  now a Boltzmann framework in order to describe the statistical evolution of a system of agents ruled by binary interactions, \cite{APTZ,PTa}.

Let $\mu(x,t)$ denote the kinetic density of agents in position $x\in\Omega$ at time $t\geq0$, such that  the total mass is normalized
$$\rho(t) = \int_{\Omega} \mu(x,t) \ dx  = 1,$$
and the time evolution of the density $\mu$ is given as a balance between the bilinear gain and loss of the agents position due to the binary interaction. In a general formulation, we assume that two agents have positions $x, y\in\Omega$ and modify their positions according to the following rule 
\begin{equation}
\begin{aligned}\label{eq:bin}
x^*= \,& x        + \alpha P(x,y)(y-x)   + \alpha U_\alpha(x,y,t)   +  \sqrt{2\alpha}\xi,       
\\
y^* =\, & y+ \alpha P(y,x)(x-y)         +  \alpha U_\alpha(y,x,t)   +  \sqrt{2\alpha}\zeta\,,
\end{aligned}
\end{equation}
where $(x^*,y^*)$ are the post-interaction positions, the parameter $\alpha$ measures the influence strength of the different terms, $(\xi,\zeta)$ is a vector of i.i.d. random variables with a symmetric distribution $\Theta(\cdot)$ with  zero mean and variance $\sigma$, and $U_\alpha(x,y,t)$ indicates the forcing term due to the control dynamics.

 We consider now a kinetic model for the  evolution of  the density $\mu=\mu(x,t)$ of agents with $x \in \R^d$  at time $t\geq 0$ and ruled by the following Boltzmann-type equation
\begin{align}\label{eq:Boltz}
\pa_t \mu(x,t) = Q_{\alpha}(\mu,\mu)(x,t),
\end{align}
where the interaction operator $Q_{\alpha}(\mu,\mu)$ in \eqref{eq:Boltz}, accounts the loss and gain of agents in position $x$ at time $t$, as follows
\begin{align}\label{eq:QBoltzB}
Q_{\alpha}(\mu,\mu)(x,t) = \Exp\left[\int_{\Omega}\left(\mathcal{B}_* \frac{1}{\mathcal{J}_\alpha}\mu(x_*,t) \mu(y_*,t) - \mathcal{B}\mu(x,t)\mu(y,t)\right)\,dy\right],
\end{align}
where $(x_*,y_*)$ are the pre-interaction positions that generate arrivals $(x,y)$. 
The bilinear operator $Q_\alpha(\cdot,\cdot)$ includes the expectation value with respect to $\xi^x$ and $\xi^y$,
 while $\mathcal{J}_\alpha$ represents the Jacobian of the transformation  $(x,y)\to(x^*,y^*)$, described by \eqref{eq:bin}. 
 Here $\mathcal{B}_*=\mathcal{B}_{(x_*,y_*)\to(x,y)} $ and  $\mathcal{B}=\mathcal{B}_{(x,y)\to(x^*,y^*)} $ are the transition rate functions. More into the details we take into account
 \[
\mathcal{B}_{(x,y)\to(x^*,y^*)} = \eta\chi_\Omega(x^*)\chi_\Omega(y^*), 
 \]
as the functions with an interaction rate $\eta>0$, and where $\chi_\Omega$ is the characteristic function of the domain $\Omega$. 
Note that in this case the transition functions depends on the relative position, similarly to \cite{T}, as we introduced a bounded domain $\Omega$ into the dynamics. 
A major simplification occurs in the case the bounded domain is preserved by the binary interactions itself, therefore  the transition is constant and the interaction operator \eqref{eq:QBoltzB} reads
\begin{align}\label{eq:QBoltz}
Q_{\alpha}(\mu,\mu)(x,t) = \eta\Exp\left[\int_{\Omega}\left( \frac{1}{\mathcal{J}_\alpha}\mu(x_*,t) \mu(y_*,t) - \mu(x,t)\mu(y,t)\right)\,dy\right].
\end{align}
In \cite{AHP,T} authors showed that in opinion dynamics binary interactions are able to preserve the boundary, according to the choice of a small support of the symmetric random variable $\xi$ and introducing a suitable function $D(x)$ acting as a local weight on the noise in \eqref{eq:bin}. 

In the next section we will perform the analysis of this model in the simplified case of $\Omega = \mathbb{R}^d$ and  constant rate of interaction $\eta$.

\begin{remark}
Note that the binary dynamics \eqref{eq:bin} is equivalent to the Euler--Maruyama discretization for the equation \eqref{eq:HKu} in the two agents case
\begin{equation}
\begin{aligned}\label{eq:stodisc_micr2}
x_i^{m+1} = \,& x_i^m        + \frac{\Delta t}{2} P(x_i^m,x_j^m)(x_j^m-x_i^m)    +  {\Delta t} U^m_{ij}    +  \sqrt{2\sigma}\Delta B^m_i,       
\\
x_j^{m+1}=\, & x_j^m+  \frac{\Delta t}{2} P(x_j^m,x_i^m)(x_i^m-x_j^m)   +  {\Delta t} U^m_{ji} +  \sqrt{2\sigma}\Delta B^m_j,
\end{aligned}
\end{equation}
where we impose that  $\alpha = \Delta t/2$, $\alpha U_\alpha(x_i,x_j) = \Delta t U^m_{ij}$, and $\sqrt{2\alpha}\xi= \sqrt{2\sigma}\Delta B^m_i$ is a random variable normally distributed with zero mean value and variance $\Delta t$,  for $\Delta B^m_i$ defined as the  $\Delta B^m_i=B_i(t_{m+1})-B_i(t_m)$.
\end{remark}

\subsubsection{The quasi-invariant limit}
We consider now the Boltzmann operator \eqref{eq:QBoltz} in the case $\Omega = \R^d$, and in order to obtain a more regular description we introduce the so-called {\em quasi-invariant interaction limit},  whose basic idea is considering a regime where interactions strength is low and frequency is high. This technique, analogous to the grazing collision limit in plasma physics, has been thoroughly studied in \cite{VILL} and specifically for first order models in \cite{CPT,T}, and allows to pass from Boltzmann equation \eqref{eq:Boltz} to a mean field equation of the Fokker-Planck-type, \cite{AHP, APZa}. 
In order to state the main result we start fixing some notation and terminology. 
\begin{definition}[Multi-index]
For any $a \in \mathbb{N}^d$ we set $|a| = \sum^d_{i = 1} a_i$, and for any function $h \in C^q(\R^d \times \R^d,\R)$, with $q\geq0$ and any $a \in \mathbb{N}^d$ such that $|a| \leq q$, we define for every $(x,v) \in \R^d \times \R^d$
\begin{align*}
\partial^{a}_x h(x) := \frac{\partial^{|a|} h}{\partial^{a_1}x_1 \cdots \partial^{a_d}x_d} (x),
\end{align*}
with the convention that if $a= (0, \ldots, 0)$ then $\partial^{a}_x h(x) := h(x)$.
\end{definition}
\begin{definition}[Test functions] We denote by $\mathcal{T}_{\delta}$ the set of compactly supported functions $\varphi$ from $\R^{d}$ to $\R$ such that for any multi-index $a  \in \mathbb{N}^d$ we have,
\begin{enumerate}
\item 
if $|a| < 2$, then $\partial^{a}_x \varphi (\cdot)$ is continuous for every $x \in \R^d$;
\item 
 if $|a| = 2$, then there exists $C > 0$ such that,
$\partial^{a}_x \varphi (\cdot)$ is uniformly H\"older continuous of order $\delta$ for every $x \in \R^d$ with H\"older bound $C$, that is for every $x,y \in \R^d$
\begin{align*}
\left\| \partial^{a}_x \varphi (x) - \partial^{a}_x \varphi (y) \right\| \leq C \left\| x - y \right\|^{\delta},
\end{align*}
and $\|\partial^{a}_x \varphi (x)\| \leq C$ for every $x \in \R^{d}$.
\end{enumerate}
\end{definition}

\begin{definition}[$\delta$-weak solution] 
Let $T > 0$, $\delta > 0$, we call a \emph{$\delta$-weak solution} of the initial value problem for the equation \eqref{eq:Boltz}, with initial datum $\mu^0=\mu(x,0) \in \mathcal{M}_0(\R^{d})$ in the interval $[0,T]$, if $\mu \in L^2([0,T], \mathcal{M}_0(\R^{d}))$ such that, $\mu(x,0) = \mu^0(x)$ for every $x \in \R^{d}$,
and there exists $R_T > 0$ such that $\textrm{supp}(\mu(t)) \subset B_{R_T}(0)$ for every $t \in [0,T]$ and $\mu$ satisfies the weak form of the equation \eqref{eq:Boltz}, i.e.,
\begin{align}\label{eq:wBoltz}
\frac{d}{d t} \left \langle \mu, \varphi \right\rangle  =  \left\langle Q_\alpha(\mu,\mu),\varphi \right\rangle ,
\end{align}
for all $t \in (0,T]$ and all $\varphi \in \mathcal{T}_{\delta}$, where  
\begin{align}\label{eq:wQBoltz}
\left\langle Q_\alpha(\mu,\mu),\varphi \right\rangle & = \Exp\left[\int_{\mathbb{R}^{2d}} \eta\left(\varphi(x^*) - \varphi(x) \right)\mu(x)\mu(y) \ dx \ dy\right].
\end{align}
\end{definition}
Moreover, we assume that
\begin{itemize}
\item[$(a)$]  the system \eqref{eq:bin} constitutes invertible changes of variables from $(x,y)$ to $(x^*,y^{*})$;
\item[$(b)$]  there exists an integrable function $K(x,y,t)$ such that the following limit is well defined
\begin{align}\label{eq:limK}
\lim_{\alpha\to 0}U_\alpha(x,y,t) = K(x,y,t).
\end{align}
In the case of instantaneous control of type  \eqref{eq:IC}, we can explicitly give an expression to the limit as  $K(x,y,t) = (x_d-x)/\gamma $.
\end{itemize}
We state the following theorem.

\begin{thm}\label{thm:grazing}
Let us fix a control  $U_\alpha\in\U$ and $\alpha\geq0$, and $T > 0$, $\delta > 0$, $\varepsilon>0$, and assume that density $\Theta\in\P_{2+\delta}(\R^d)$ and the function $P(\cdot,\cdot)\in L^q_{loc}$ for $q = 2, 2+\delta$ and for every $t\geq0$. We consider a $\delta$-weak solution $\mu$ of equation \eqref{eq:Boltz} with initial datum $\mu_0(x)$. 
Thus introducing the following scaling 
\begin{equation}\label{eq:scaling}
\alpha = \varepsilon, \qquad \eta ={1/\varepsilon},
\end{equation}
for the binary interaction \eqref{eq:bin} and  defining by $\mu^\varepsilon(x,t)$ a solution for the scaled equation  \eqref{eq:Boltz},
for $\varepsilon\to0$  $\mu^\varepsilon(x,t)$ converges pointwise, up to a subsequence, to  $\mu(x,t)$  where $\mu$ satisfies  the following Fokker-Planck-type equation,
\begin{align}\label{eq:FP}
\pa_t \mu + \nabla \cdot\left((\mathcal{P}[\mu] + \bigK[\mu])\mu\right)= \sigma\Delta \mu,
\end{align}
with initial data $\mu_0(x)=\mu(x,0)$ and 
where $\mathcal{P}$ represents the interaction kernel  \eqref{eq:kernelP} and $f(x,t)$ is the control.
\begin{align}\label{eq:kernelK}
\mathcal{K}[\mu](x,t)  = \int_{\mathbb{R}^d}K(x,y,t)\mu(y,t)\,dy. 
\end{align} 
with $K(x,y,t)$ defined as in \eqref{eq:limK}.
\end{thm}

\begin{proof}

$\bullet$ {\bf Taylor approximation.} We consider the weak formulation of the Boltzmann equation \eqref{eq:wBoltz} and    
we expand $\varphi(x^{*})$ inside the operator \eqref{eq:wQBoltz} in Taylor series of $x^* - x$ up to the second order, obtaining 
\begin{align}\label{eq:TwQBoltz}
&\left\langle Q_\alpha(\mu,\mu),\varphi \right\rangle  = T^\varphi_1+T^\varphi_2+ R_1^{\varphi},
\end{align}
where the first and second order terms are
\begin{align}
&T^\varphi_1:= \eta\Exp\Bigg[\int_{\mathbb{R}^{2d}} \nabla \varphi(x) \cdot \left(x^{*} - x\right)\mu(x)\mu(y) \, dxdy\Bigg], \\
&T^\varphi_2 :=  \frac{\eta}{2}\Exp\Bigg[\int_{\mathbb{R}^{2d}} \left(\sum^{d}_{i,j = 1} \partial^{(i,j)}_{x} \varphi(x) \left(x^{*} - x\right)_i\left(x^{*} - x\right)_j\right) \mu(x)\mu(y) \, dxdy \Bigg],
\end{align}
and $R_1^{\varphi}(\epsi)$ is the reminder of the Taylor expansion, with a form
\begin{align*}
R_1^{\varphi}&:= \frac{\eta}{2}\Exp\Bigg[\int_{\mathbb{R}^{2d}} \left(\sum^{d}_{i,j = 1} \left(\partial^{(i,j)}_{x} \varphi(x) - \partial^{(i,j)}_{x} \varphi(\overline{x})\right) \left(x^{*} - x\right)_i\left(x^{*} - x\right)_j\right)\mu(x)\mu({y}) \, dxdy \Bigg],
\end{align*}
with $\overline{x} := (1-\theta) x^* + \theta x$, for some $\theta \in [0,1]$. 
By using the relation given by the scaled interaction rule \eqref{eq:bin}, i.e.
$$x^* - x = \alpha F_\alpha(x,y) +\sqrt{2\alpha}\xi$$
where for the sake of brevity  we denoted $F_\alpha(x,y) := P(x,y)(y-x) +U_\alpha(x,y)$. Note that from the hypothesis it follows that $F_\alpha\in L^q_{loc}$. Thus  we obtain 
\begin{align*}
T^\varphi_1 &= \eta\Exp\Bigg[\alpha\int_{\mathbb{R}^{2d}} \nabla\varphi(x) \cdot \left(F_\alpha(x,y) +\sqrt{2/\alpha}\ \xi\right)\mu(x)\mu(y)\, dxdy\Bigg]\\&
= \eta\alpha\int_{\mathbb{R}^{2d}} \nabla\varphi(x) \cdot F_\alpha(x,y)\mu(x)\mu(y)\, dxdy
\end{align*}
where the  noise term, $\xi$ is canceled out since it has zero mean. For the same reason in the second order term $T^\varphi_2$ all mixed product between $F_\alpha$ and $\xi$ vanish, the same hold for all the crossing terms $\xi_{i} \xi_{j}$ since $\xi_{i}$ are supposed to be independent variables. Hence the only contribution we have reads
\begin{align*}
T^\varphi_2 &= \frac{\eta}{2}\Exp\Bigg[ \int_{\mathbb{R}^{2d}}\alpha^2\left(\sum^{d}_{j = 1} \partial^{(j,j)}_{x} \varphi(x)\left(F_\alpha(x,y)_j\right)^2 \right)+ \left(\sum^{d}_{j = 1} \partial^{(j,j)}_{x} \varphi(x)\left(2\alpha\xi_{j}^2\right)\right) \mu(x)\mu(y) \, dxdy \Bigg]\\
&= \eta\alpha\int_{\mathbb{R}^{2d}} \sigma\Delta\varphi(x)\mu(x)\mu(y) \, dxdy +\frac{\eta\alpha^2}{2}\int_{\mathbb{R}^{2d}} \left(\sum^{d}_{j = 1} \partial^{(j,j)}_{x} \varphi(x)\left(F_\alpha(x,y)_j\right)^2 \right)\mu(x)\mu(y) \, dxdy,\\
&=: T^\varphi_{22} + R_{2}^\varphi.
\end{align*}
$\bullet$ {\bf Quasi-invariant limit.}  We now introduce the scaling \eqref{eq:scaling}, for which we can substitute in the previous equations, $\eta\alpha=1$ and  $\eta\alpha^2=\epsi$, thus we have that terms $T_1^\varphi$ and $T_{22}^\varphi$ represent the leading order and $R^\varphi(\epsi) := R^\varphi_1+R^\varphi_2$ a reminder, so we can recast the scaled expression \eqref{eq:TwQBoltz} as follows
\begin{equation}\label{eq:TwQBoltz2}
\int_{\R^{2d}}\left(\nabla\varphi\cdot F_\epsi(x,y)+\sigma\Delta\varphi(x)\right) \mu(x)\mu(y) \, dxdy+ R^\varphi(\epsi).
\end{equation}
Let us now consider the limit $\varepsilon \rightarrow 0$,  assuming that for every $\varphi \in \mathcal{T}_{\delta}$
\begin{align} \label{eq:limRest}
\lim_{\varepsilon \rightarrow 0} R^\varphi(\epsi) = 0
\end{align}
holds true, we have thanks to \eqref{eq:limK}  and  \eqref{eq:TwQBoltz2} that the weak scaled Boltzman equation \eqref{eq:wBoltz} converges pointwise to the Fokker-Planck-type equation \eqref{eq:FP} as follows
\begin{align} \label{eq:wFP}
\frac{d}{dt} \left \langle \mu, \varphi \right\rangle  =  \left\langle \mu, \nabla \varphi \cdot (\mathcal{P}\left[\mu\right]+\bigK[\mu]) + \sigma \Delta \varphi \right\rangle,
\end{align}
where  the operators $\mathcal{P}[\mu]$ and  $\mathcal{K}[\mu]$ are defined in \eqref{eq:kernelP} and \eqref{eq:kernelK}. Since $\varphi$ has compact support, equation \eqref{eq:wFP} can be revert in strong form by means of integration by parts, we eventually obtain system \eqref{eq:FP}.

$\bullet$ {\bf Estimates for the reminder.}
In order to conclude the proof it is sufficient to show that the limit \eqref{eq:limRest} for $R^\varphi(\epsi)$ vanishes.
From the definition of  $\overline{x}$ it follows that $\left\| \overline{x} -x\right\| \leq \left\| x^* - x \right\|$, then for every $\varphi \in \mathcal{T}_{\delta}$ we have
\begin{align*}
\left\|\partial^{(i,j)}_{x}\varphi(x) - \partial^{(i,j)}_{x}\varphi(\overline{x})\right\| & \leq C \left\|  \overline{x} -x\right\|^{\delta}  \leq C \left\| x^* - x \right\|^{\delta}.
\end{align*}
Hence for $R_1^\varphi$ we get
\begin{align*}
\left\|R_1^{\varphi}\right\| & \leq \frac{C}{2\varepsilon} \Exp\left[\int_{\R^{2d}} \left\| x^* - x \right\|^{2+\delta} \mu(x) \mu(y) \, dxdy\right] \\
& = \frac{C}{2} \epsi^{1+\delta} \Exp\left[\int_{\R^{2d}} \left\| F_\epsi(x,y) + \sqrt{2/\epsi}\ \xi \right\|^{2+\delta} \mu(x) \mu(y)  \, dxdy \right]
\end{align*}
from the inequality $|a+b|^{2+\delta}\leq2^{2+2\delta}(|a|^{2+\delta}+|b|^{2+\delta})$ for some $a,b$  we obtain
\begin{align*}
\left\|R_1^{\varphi}\right\| & \leq  2^{1+2\delta}C\left(\epsi^{1+\delta} \int_{\R^{2d}} \left\| F_\epsi(x,y)\right\|^{2+\delta} \mu(x) \mu(y) \, dxdy+ 2^{1+\delta/2}\epsi^{\delta/2}\Exp\left[ \left\| \xi \right\|^{2+\delta}\right]\right).
\end{align*}
Analogous computation can be yield for $R_2^\varphi$ for which we have the following inequality
\begin{align*}
\left\|R_2^{\varphi}\right\| & \leq  \frac{\epsi C}{2} \int_{\R^{2d}} \left\| F_\epsi(x,y)\right\|^{2} \mu(x) \mu(y)  \, dxdy.
\end{align*}
Since $F_\epsi\in L^q_{loc}$ for $q=2,2+\delta$ and $\Theta\in\P_{2+\delta}(\R^d)$ we can conclude that for $\epsi\to0$ the limit \eqref{eq:limRest} holds true. 
\end{proof}

\begin{remark}
Note that in the case $U_\alpha(x,y,t)=U_\alpha(x,t)$, namely if the feedback control depends only by the position $x$ of the agents at time $t$, then the kernel $\bigK[\mu](x,t)$ reduces to $K(x,t)$. This observation holds also if we consider a sampling from the optimal control, i.e. $U_\alpha(x,y,t) = f(x,t)$, thus equation \eqref{eq:FP} becomes exactly the original equation \eqref{eq:HKu}.
\end{remark}
\section{Numerical methods}\label{sec:num}

In this section we are concerned with the development of numerical methods for the mean field optimal control problem \eqref{eq:HKu}-- \eqref{eq:funcJu}. First we present direct simulation Monte Carlo methods for the constrained Boltzmann-type model \eqref{eq:Boltz}, and discuss the implementation of the binary feedback controllers introduced in Section \ref{sec:bcd}. Next, we describe a sweeping algorithm based on the iterative solution of the optimality system, \eqref{eq:MFOC}--\eqref{eq_fpsi}.

\subsection{Asymptotic constrained binary algorithms}
One of the most common approaches to solve Boltzmann-type equations is based on Monte Carlo methods. Let us consider the initial value problem given by the equation \eqref{eq:Boltz}, in the grazing interaction regime \eqref{eq:scaling}, with initial data $\mu(x,t=0)=\mu_0(x)$, as follows
\begin{equation}\begin{cases}\vspace{0.5em}\label{eq:Coll}
\dfrac{d}{dt}\mu(x,t) =  \dfrac{1}{\epsi}\left[{Q}_\epsi^{+}(\mu,\mu)(x,t)-\mu(x,t)\right], \\
\mu(x,0)  =\mu_0(x).
\end{cases}\end{equation}
Here we have made explicit the dependence of  the interaction operator $Q_\epsi(\cdot,\cdot)$ on the frequency of interactions $1/\varepsilon$, and decomposing it into its gain and loss parts according to \eqref{eq:QBoltz}. With $Q^{+}_\varepsilon(\cdot,\cdot)$ we denote the gain part, which accounts the density of agents gained at position $x$ after the binary interaction  \eqref{eq:bin}.
 
We tackle the Boltzmann-type equation \eqref{eq:Coll} by means of a binary interaction algorithm \cite{APb, PTa}, where the basic idea is to solve the binary exchange of information described by \eqref{eq:bin}, under the grazing interaction scaling \eqref{eq:scaling}, in order to obtain in the limit an approximate solution of the mean field equation  \eqref{eq:FP}. Note that the consistency of this procedure is given by Theorem \ref{thm:grazing}.

 Let us now consider a time interval $[0,T]$ discretized in $M_{tot}$ intervals of size $\Delta t$. We denote by $\mu^m$ the  approximation of $\mu(x,m\Delta t)$, thus the first order forward scheme of the scaled Boltzmann-type equation \eqref{eq:Coll} reads
\begin{equation}\label{eq:MCBoltz}
        \mu^{m+1}=\left(1-\frac{\Delta t}{\varepsilon}\right)\mu^{m}+\frac{\Delta t}{\varepsilon}{{Q}_\varepsilon^{+}(\mu^m,\mu^m)},
\end{equation}
 where, since $\mu^m$ is a probability density, thanks to mass conservation, and also $Q_\varepsilon^{+}(\mu^m,\mu^m)$ is a probability density. Under the restriction $\Delta t\leq\varepsilon$, 
 $\mu^{m+1}$ is a probability density, since it is a convex combination of probability densities. 

From a Monte Carlo point of view the equation \eqref{eq:MCBoltz} can be interpreted as follows: an individual with position $x$ will not interact with other individuals with probability $1-\Delta t/\varepsilon$ and it will interact with others with probability $\Delta t/\varepsilon$ according to the interaction law stated by $Q_\varepsilon^{+}(\mu^m,\mu^m)$. 
Note that, since we aim at small values of $\varepsilon$ and we have to fulfill the condition $\Delta t\leq\epsi$, the natural choice is to take $\Delta t=\varepsilon$. At every time step, this choice maximizes the number of interactions among the agents.

For the numerical treatment of the operator $Q_\varepsilon^{+}(\mu^m,\mu^m)$, we have to account in every interaction the action of the feedback control. 
In the case of {\em instantaneous control} this can be evaluated directly, for example in the case of quadratic functional defining the scaling version of \eqref{eq:IC2} as
\[
U_\epsi(x,y,t)  = \frac{1}{\gamma+\epsi}\left((x_d-x)+\alpha P(x,y)(y-x)\right).
\] 

On the other hand, the realization of the optimal feedback controller in the {\em  finite horizon setting} requires the numerical approximation of the Bellman equation \eqref{eq:DP}. This approximation is performed offline and only once, previous to the simulation of the mean field model. For a state space of  moderate dimension, such as in our binary model, several numerical schemes for the approximation of Hamilton-Jacobbi-Bellman equations are available, and we refer the reader to \cite[Chapter 8]{FFbook} for a comprehensive description of the different available techniques. Since the binary model is already introduced in discrete time, a natural choice is to solve eq. \eqref{eq:DP} by means of an sequential semi-Lagrangian scheme, following the same guidelines as in the recent works \cite{AFK15,KKK16,F16}. Once the value function has been approximated, online feedback controllers can be implemented through the evaluation of the optimality condition \eqref{eq:ochj}.

We report in Algorithm \ref{ANMC}  a stochastic procedure to solve \eqref{eq:MCBoltz}, based on Nanbu's method for plasma physics, \cite{APb, BN}.
\begin{algorithm}\caption{\em Asymptotic constrained binary algorithm}\label{ANMC}
  \begin{enumerate}
  \item[\texttt{0.}] Pre-compute the feedback control $U_\epsi(x,y,t)$ on an appropriate discretized grid of the domain $\Omega\times[0,T]$.
 	\vspace{-0.2cm}
  \item[\texttt{1.}] 
  Given $N_s$ samples $\left\{x^0_k\right\}_{k=1}^{N_s}$, from the initial distribution $\mu_0(x)$;
   	\vspace{-0.2cm}
  \item[]
  \texttt{for} $m=0$ \texttt{to} $M_{tot}-1$
  	\vspace{-0.2cm}
  \begin{enumerate}
  \item[\texttt{a.}] set $N_c = \textsc{Iround}({N_s}/{2})$;
  \item[\texttt{b.}] select $N_c$ random pairs $(i,j)$ uniformly without repetition among all possible pairs of individuals at time level $t_m$;
  \item[\texttt{c.}] evaluate $P(x_i,x_j), P(x_j,x_i)$ and $U_\epsi(x_i,x_j,t_m), U_\epsi(x_j,x_i,t_m)$;
  \item[\texttt{d.}] compute the post-interaction position $x_i^*$, $x_j^*$ for each pair $(i,j)$ using relations \eqref{eq:bin} and $\xi_i,\xi_j$ sampled from a normal distribution $\mathcal{N}(0,\sigma)$;
  \item[\texttt{e.}] set $x_i^{n+1}=x_i^{*}$, $x_j^{n+1}=x_j^{*}$.
  \end{enumerate}
  	\vspace{-0.2cm}
\item[]\texttt{end for}
  \end{enumerate}
  	\vspace{-0.2cm}
\end{algorithm}\medskip
\vspace{-0.2cm}

Where function $\textsc{Iround}(\cdot)$ denotes the integer stochastic rounding defined as
\[
\textsc{Iround}(x)=
\begin{cases}
[x]+1,& \zeta < x-[x],\\
[x],&\hbox{elsewhere}
\end{cases}
\]
with $\zeta$ a uniform $[0,1]$ random number and $[\cdot]$ the integer part.

\begin{remark}[Efficency]
In general, computing the interactions among a multi-agent system is a procedure of quadratic cost with respect to the number of agents, since every agent needs to evaluate its influence with every other. Note that with the proposed algorithm this cost becomes linear with respect to the number of samples introduced $O(N_s)$, since only binary interactions are accounted. A major difference compared to standard algorithms for Boltzmann equations is the way in which particles are sampled from $Q_{\varepsilon}^+(\mu^m,\mu^m)$ which does not require the introduction of a space grid \cite{BN}.
\end{remark}
\begin{remark}[Accuracy]
The  choice $\Delta t=\varepsilon$ is optimal if $\epsi$ is of the order of $O({N_s}^{-1/2})$. Indeed, the accuracy of the method will not increase for smaller values of $\Delta t$, because the numerical error is dominated by the fluctuations of the Monte Carlo method. For further details we refer to \cite{APb,PTa}.
\end{remark}

\subsection{Numerical approximation of the optimality conditions}
As shown in Section \ref{sec:firstorderopt}, the solution of the mean field optimal control problem \eqref{eq:MFOC}-\eqref{eq:Jfun} satisfies the optimality system
\begin{align}
\partial_t\mu &=- \nabla \cdot ((\bigF[\mu] + f)\mu) +\sigma\Delta \mu\,,\label{eq:forward}\\
- \partial_t\psi &=\frac{1}{2}|x-x_d|^2 +\gamma \CLambda(f) + \nabla \psi \cdot f + \sigma \Delta \psi\nonumber\\&- \frac12\int_\om \lt(P(x,y)\nabla \psi(x,t) - P(y,x)\nabla \psi(y,t)\rt)\cdot (y-x)\mu(y,t)\,dy\,,\label{eq:backward}\\
\nabla \CLambda (f) &= \frac1\gamma \nabla \psi\,,\quad \mu(x,0)=\mu_0(x)\,,\quad\psi(x,T)=0.\label{eq:conditions}
\end{align}
\paragraph{Forward equation.} 
In order to solve equation \eqref{eq:forward}, we  consider a first order forward scheme the time evolution and the Chang-Cooper scheme for the space discretization, \cite{CC}. 
The  formulation is based on the finite volume approximation of the density $\mu$ and $f$. Defining the operator $\mathcal{G}[\mu,f] := \F[\mu,f] +\sigma\nabla \mu$, with $\F[\mu,f] = \QF[\mu] + f$, then we can write in the one-dimensional domain $[-L,L]$ the (semi)-discretized equation \eqref{eq:forward} as
\begin{equation}\label{eq:GG}
\frac{d}{dt} \mu_i(t) = \frac{\mathcal{G}_{i+1/2}[\mu,f]-\mathcal{G}_{i-1/2}[\mu,f]}{\delta x},\quad  \textrm{ with } \quad \mu_i(t) =\frac{1}{\delta x}\int^{x+1/2}_{x-{1/2}} \mu(x,t) \ dx,
\end{equation}
where we have introduced the uniform grid $x_{i}=-L+i\delta x$, $i=0,\ldots,N,$ with $\delta x = 2L/N$, and denoted by $x_{i \pm 1/2}=x_i \pm \delta x/2$.  Thus, the operator $\mathcal{G}_{i+1/2}[\mu,f]$ in the case of constant diffusion $\sigma$ reads
\be
\begin{split}
\label{eq:flux}
\mathcal{G}_{i+1/2}[\mu,f]=&\left((1-\theta_{i+1/2})\mu_{i+1}+\theta_{i+1/2}\mu_i\right)\mathcal{F}[\mu_{i+1/2}, f_{i+1/2}] +\frac{\sigma (\mu_{i+1}-\mu_i)}{\delta x},
\end{split}
\ee  
where the weights $\theta_{i+1/2}$ are in general depending on the solution and the parameters of equation \eqref{eq:forward}. Hence the flux functions are defined as a combination of upwind and centered discretizations, and such that for $\sigma = 0$ the scheme reduces to an upwind scheme, i.e. $\theta_{i+1/2} = 0$. The choice of the weights is the key point of the scheme \eqref{eq:GG}, which allows to preserve steady state solutions and the non-negativity of the numerical density. We refer to \cite{APZd, BD10,CC} for the details on the properties and analysis of the Chang-Cooper scheme for similar Fokker-Planck models and to \cite{SAB16}, and references therein, for applications to control problems. 

Alternatively, scheme \eqref{eq:MCBoltz} furnishes a consistent method to solve the forward equation \eqref{eq:forward}, which we expect to be more efficient for problems with high dimensionality, since it relies on a stochastic evaluation of the nonlocal operator $\mathcal{P}[f]$.

\paragraph{Backward equation.} The main difficulty of the integro-differential advection-reaction-diffusion equation \eqref{eq:backward} resides on the efficient approximation of the integral term. We follow a finite difference approach, which we describe in the following. First, with time parameter $\delta t$ as in the forward problem, we consider the first-order  temporal approximation
\begin{align*}
- \frac{\psi^m-\psi^{m+1}}{\delta t} &=\frac{1}{2}|x-x_d|^2 +\gamma \CLambda(f^{m+1}) + \left(f^{m+1}- \frac12\int_\om P(x,y)\cdot (y-x)\mu^{m+1}\,dy\right)\cdot\nabla \psi^{m+1}\\&+ \sigma \Delta \psi^{m+1}+\frac12\int_\om \lt(P(y,x)\nabla_y \psi^{m+1}\rt)\cdot (y-x)\mu^{m+1}\,dy\,,\quad m=0,\ldots,M
\end{align*}
where $\psi^M=0$. At this level, $f$, $\mu$, and $\nabla\psi$ are treated as external data available at every discrete instance. In particular $\nabla_y$ (inside the integral) is reconstructed by numerical differentiation.  Then, the integral terms are evaluated with a Monte Carlo method generating $M_s$ samples according to the distribution $\mu$, and values of $\nabla_y\psi$ are obtained by interpolation of the reconstructed variable. The advection term is approximated with a space-dependent upwind scheme, and diffusion is approximated with centered differences.

\paragraph{Optimality condition and sweeping iteration.} Once the forward-backward system has been discretized, what remains is to establish a coupling procedure in order to find the solution of the optimality system matching both initial and terminal conditions. For this, a first possibility is to consider  the full space-time discretization of the forward-backward system, together with the optimality condition $\nabla \CLambda (f) = \frac1\gamma \nabla \psi$, and cast it as a large-scale set of nonlinear equations, which can be solved via a Newton method. This idea has been already successfully applied in the context of mean field games in \cite{ACCD12}. We pursue a different approach that has proven to be equally effective, developed in \cite{CS14}, where the authors apply a sweeping algorithm, which in our setting reads as follows.
\begin{algorithm}[H]
	\caption{\em Sweeping algorithm}\label{alg:sweeping}
		\begin{enumerate} 
              	 \item[\texttt{0.}] Given initial guess $f_0$, tolerance $tol$, and $i=0$
                 \vspace{-0.2cm}
            	  \item[\texttt{1.}] \texttt{while} $\|f_i-f_{i-1}\|\leq tol$
 		\vspace{-0.2cm}
                \begin{enumerate}
               \item[\texttt{a.}] Perform a forward solve \eqref{eq:forward} with data $f_i$ for $\mu_i$;
               \item[\texttt{b.}] Perform a backward solve \eqref{eq:backward} with data $f_i,\mu_i$, for $\psi_i$;
               \item[\texttt{c.}] Update through $\nabla \CLambda (f_{i+1}) = \frac1\gamma \nabla \psi_i$; 
               \item[\texttt{d.}] set $i=i+1$.           
\end{enumerate}
	\vspace{-0.2cm}
\texttt{end while}
\end{enumerate}
\vspace{-0.2cm}
\end{algorithm}  
Our numerical experience is consistent with what has been already reported in \cite{CS14}, in the sense that solutions satisfying the optimality system can be found after few sweeps. A more robust implementation can be obtained through a gradient-type method, as in \cite{BFMW14}.

\subsection{Numerical experiments}\label{sec:exp}
In order to validate our previous analysis we focus on  models for opinion dynamics, \cite{hekr02,SWS,PTa,T}, thus in the unidimensional case the state variable $x\in[-L,L]$ represents the agent opinion with respect to two opposite opinions $\{-L,+L\}$, and the control $f(x,t)$ can be interpreted as the strategy of a policy maker, \cite{AHP, APZa}.

Therefore we consider the following initial value problem
\begin{equation}\label{eq:MFC}
\partial_t \mu +\partial_x\left(\left(\int_{-L}^{+L}P(x,y)(y-x)\mu(y)dy +f\right)\mu\right)=\sigma\partial_x^2 \mu, \quad \mu(x,0) = \mu^0(x)
\end{equation}
with no-flux boundary conditions, and where $f$ denotes the control term, solution of 
\begin{equation}\label{eq:J_MFC}
f = \arg\min_{g\in\U}\dfrac{1}{2}\int_{0}^{T}\int_{-L}^{+L}\left(|x-x_d|^2 +\gamma g^2 \right)\mu(x,t)\ dx \ dt,
\end{equation}
where we consider a quadratic penalization of the control, i.e. $\Psi(c) = |c|^2/2$.

For different interaction kernels $P(\cdot,\cdot)$, we will study the performance of the proposed controllers $f=f(x,t)$, obtained through the following synthesis procedures: instantaneous control (IC), finite horizon (FH), and the sweeping algorithm (OC).
 
We report in Table \ref{tab:par}  the  choice of the algorithms and parameters, indicating for which method they have been used  to compute \eqref{eq:MFC}--\eqref{eq:J_MFC}.

\begin{table}[h]
\caption{Parameters choice for the various algorithms and optimization methods.}
\label{tab:all_parameters}
\begin{center}
\begin{tabular}{cccccc}
\hline
& $Algorithm$ & $N_s$ & $\epsi=\delta t$ & $\delta x$  &$tol$  \\
\hline
\hline
IC/FH & Alg \ref{ANMC}  & $5\times10^5$ & $2.5\times10^{-3}$ & $2.5\times 10^{-2}$   & $-$ \\
\hline
Uncontrolled/OC & Alg \ref{alg:sweeping}  & $-$ &$2.5\times10^{-3}$ & $2.5\times 10^{-2}$  &  $10^{-5}$   \\
\hline
\end{tabular}\label{tab:par}
\end{center}
\end{table}

\subsubsection{Test 1: Sznajd model}
We consider the Sznajd model, \cite{ANT,SWS} for which the interaction operator $P(\cdot,\cdot)$ in \eqref{eq:MFC} is defined as follows
\begin{align}\label{eq:MFSz}
P(x,y)  = \beta(1-x^2),
\end{align}
for $\beta$ a constant. Note that in this case the interaction kernel $P(\cdot,\cdot)$ models the propensity of voters to change their opinions within the domain $\Omega = [-1,1]$, and for values close to the extremal opinions $\{-1,1\}$ the influence is low, conversely for opinions close to zero the influence is high. The dynamics is such that for  $\beta>0$ \emph{concentration} of the density profile appears, whereas for $\beta<0$ \emph{separation} occurs, namely concentration around $x=1$ and $x=-1$, see \cite{ANT}.  

For our first test we fix $\beta=-1$ and we define in the time interval $[0,T]$, $T= 8$. We solve the control problem \eqref{eq:MFC}--\eqref{eq:J_MFC},  with a bivariate initial data
$\mu^0(x) := \varrho_+(x+0.75;0.05,0.5)+\varrho_+(x-0.5;0.15,1),$
where  $\varrho_+(y;a,b) := \max\{(y/b)^2-a,0\}$, with diffusion coefficient $\sigma = 0.01$, and desired state $x_d = -0.5$.

In Figure \ref{Fig_T1} we depict the final state of \eqref{eq:MFSz} at time $T = 8$ for the uncontrolled and controlled dynamics. The simulations show the concentration of the profiles around the reference position $x_d$ in presence of the control, instead in the uncontrolled case the density tends to concentrate around the boundary. The left-hand side figure refers to a penalization of the control $\gamma = 0.5$, the right-hand side figure with $\gamma = 0.05$. As expected, with smaller control penalizations, the final state is driven closer to the desired reference.

\begin{figure}[t]
\centering
\includegraphics[scale=0.35]{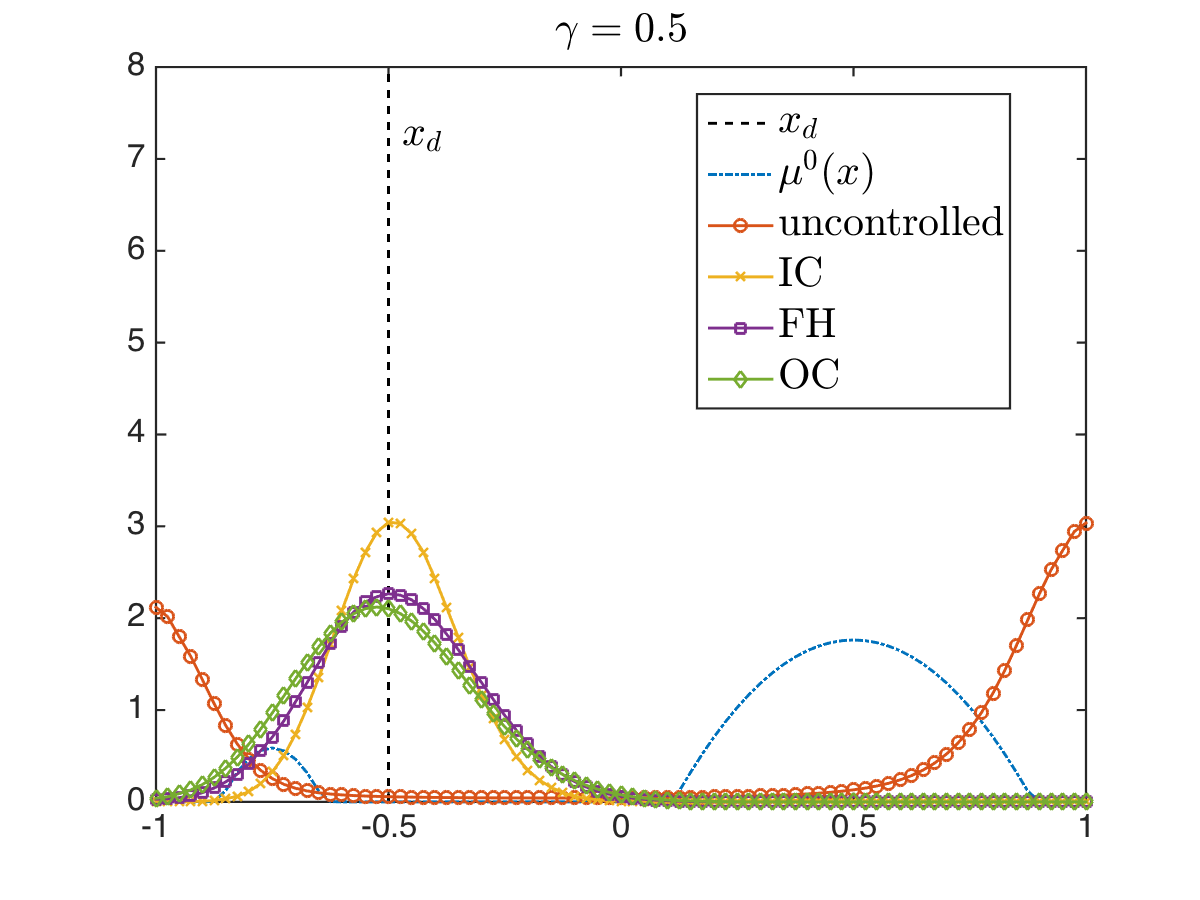}
\includegraphics[scale=0.35]{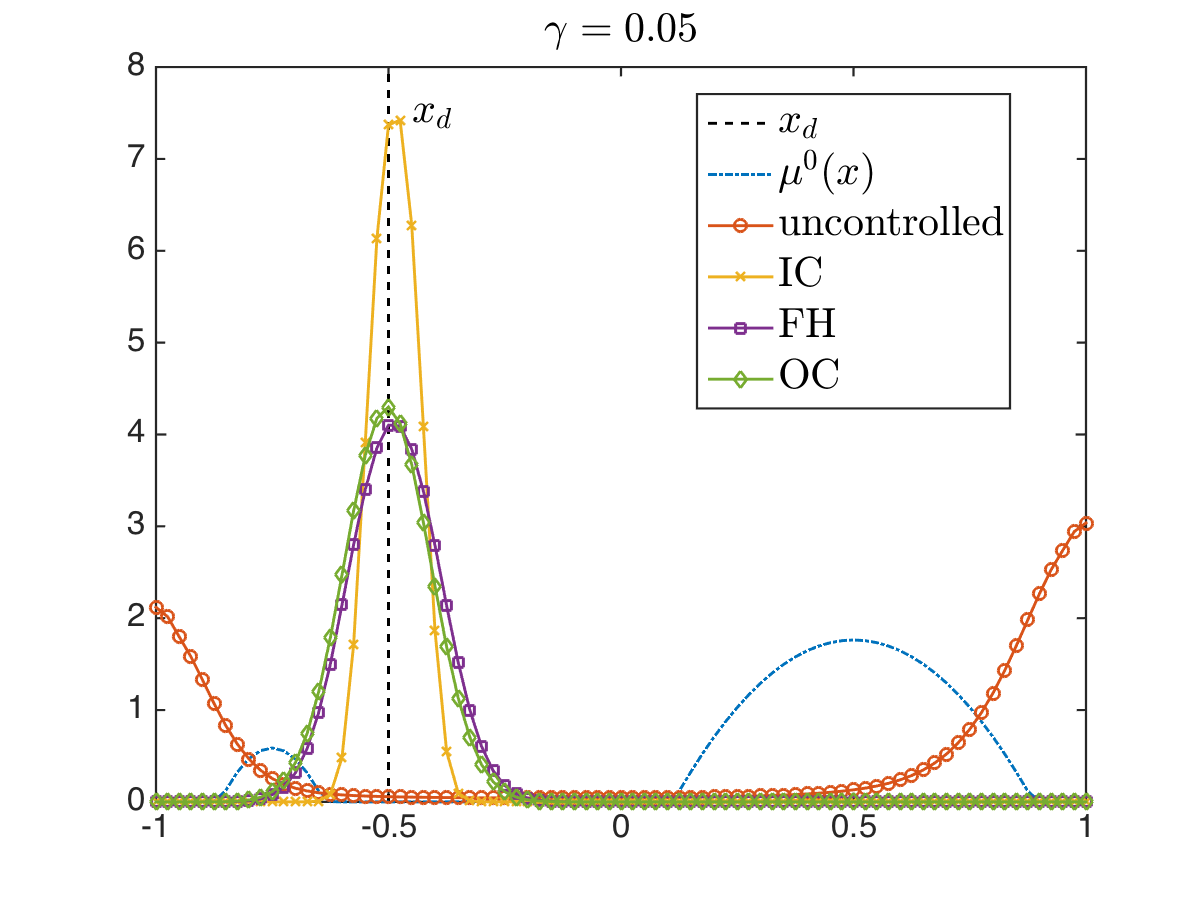}
\caption{Test \#1: Final states at time $T=8$ of the Sznajd model \eqref{eq:MFSz} for $\beta = -1$ with initial data $\mu^0(x)$. Concentration around the desired state $x_d$ is observed in presence of the controls: instantaneous control (IC), finite horizon approach (FH), optimal control (OC), separation is observed in the uncontrolled setting. Left figure $\gamma = 0.5$, right figure $\gamma = 0.05$. }
\label{Fig_T1}
\end{figure}

In Figure \ref{Fig_T1b} we depict the transient behavior of the density $\mu(x,t)$ and the control $f(x,t)$ in the $[-1,+1]\times[0,T]$ frame, respectively for $\gamma =0.5$ and $\gamma = 0.05$, and we report the values of the cost function $J(\mu,f)$ corresponding to the different methods. 
Note that that the action of the instantaneous control is almost constant in time steering the system toward $x_d$ but with the higher cost $J(\mu,f)$, on the other hand the optimal finite horizon for the binary dynamics (FH) produces a similar control with respect to the optimal control obtained by the sweeping algorithm (OC), with a small difference between the values of the cost functional.

\begin{figure}
\begin{center}
\begin{tabular}{@{}c@{\hspace{1mm}}c@{\hspace{1mm}}c@{\hspace{1mm}}c@{}}
~&&\textrm{uncontrolled}&\\
~&&\includegraphics[trim=30 10 40 20,clip,width=0.25\textwidth]{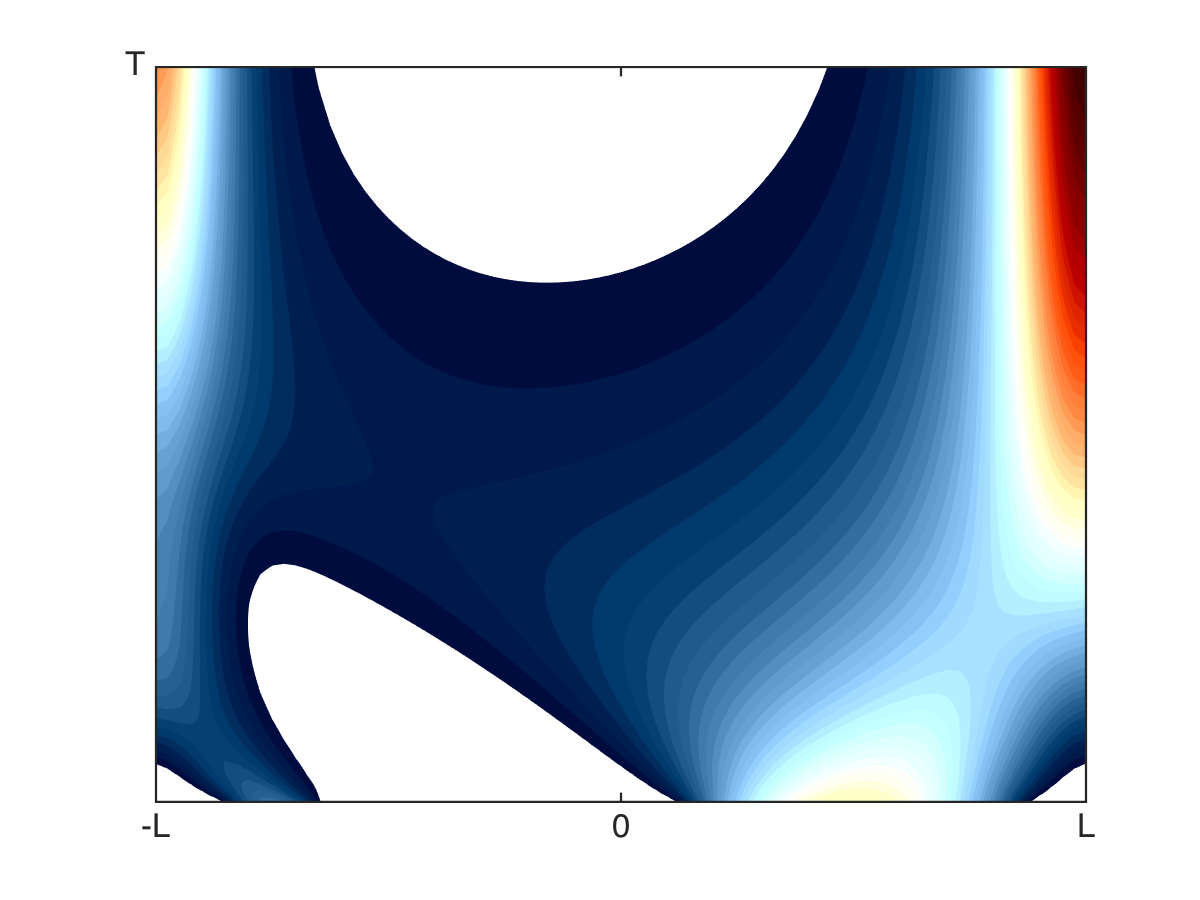}&\\
\hline
&$IC$ & $FH$ & $OC$\\
\hline
\hline
 $\gamma=0.5$&$J(\mu,f)=0.9982 $  & $J(\mu,f)=0.9467$ & $J(\mu,f)=0.9219$ \\
 \hline
\sidecap{$\mu(x,t)$} 
&
\includegraphics[trim=30 10 40 20,clip,width=0.25\textwidth]{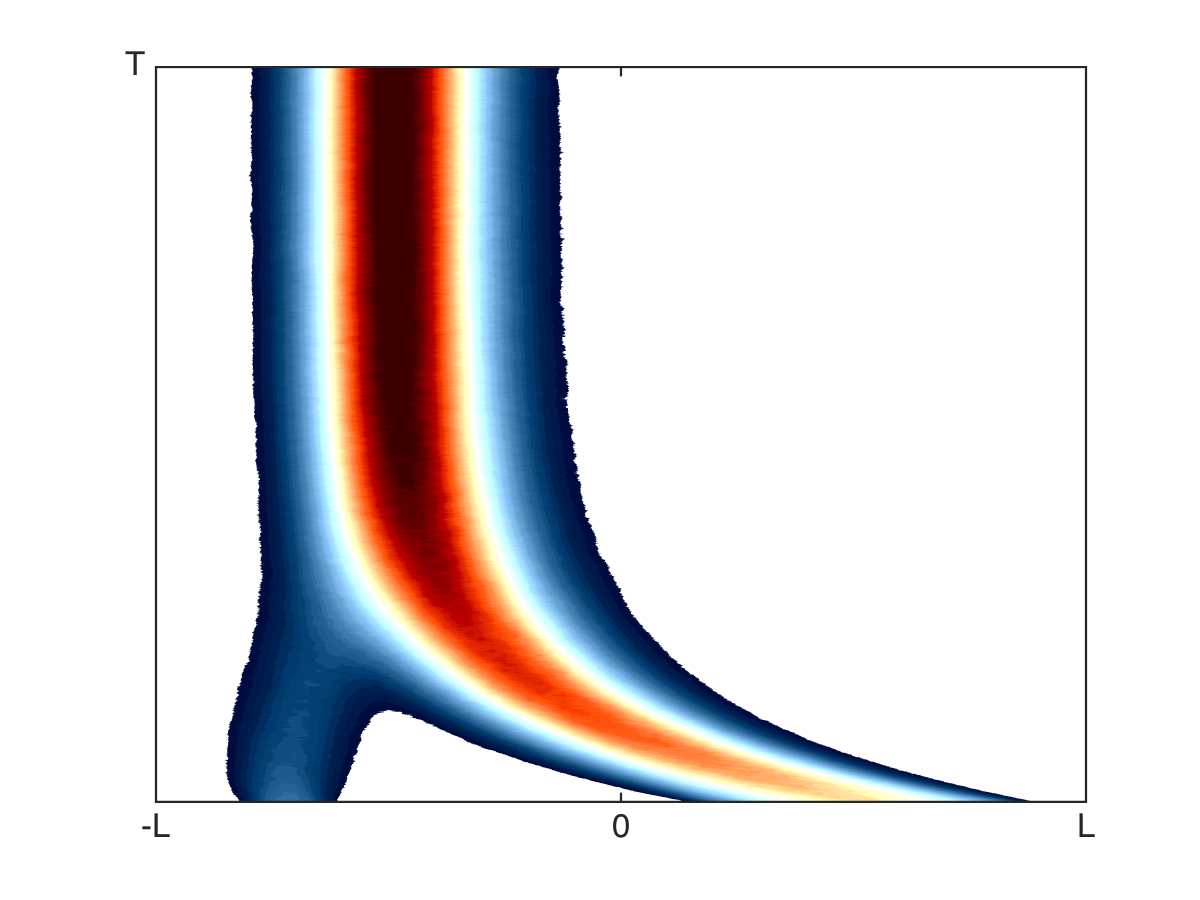}&
\includegraphics[trim=30 10 40 20,clip,width=0.25\textwidth]{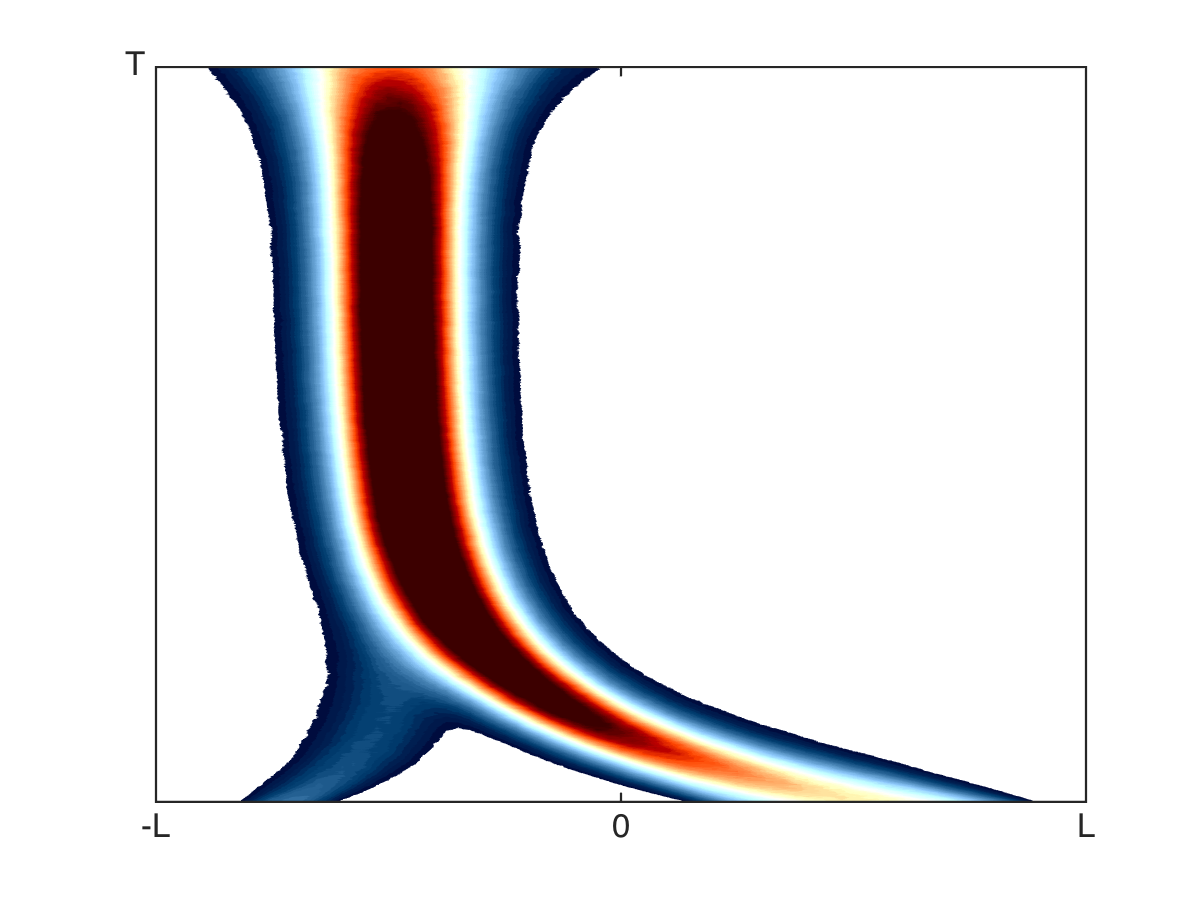}&
\includegraphics[trim=30 10 40 20,clip,width=0.25\textwidth]{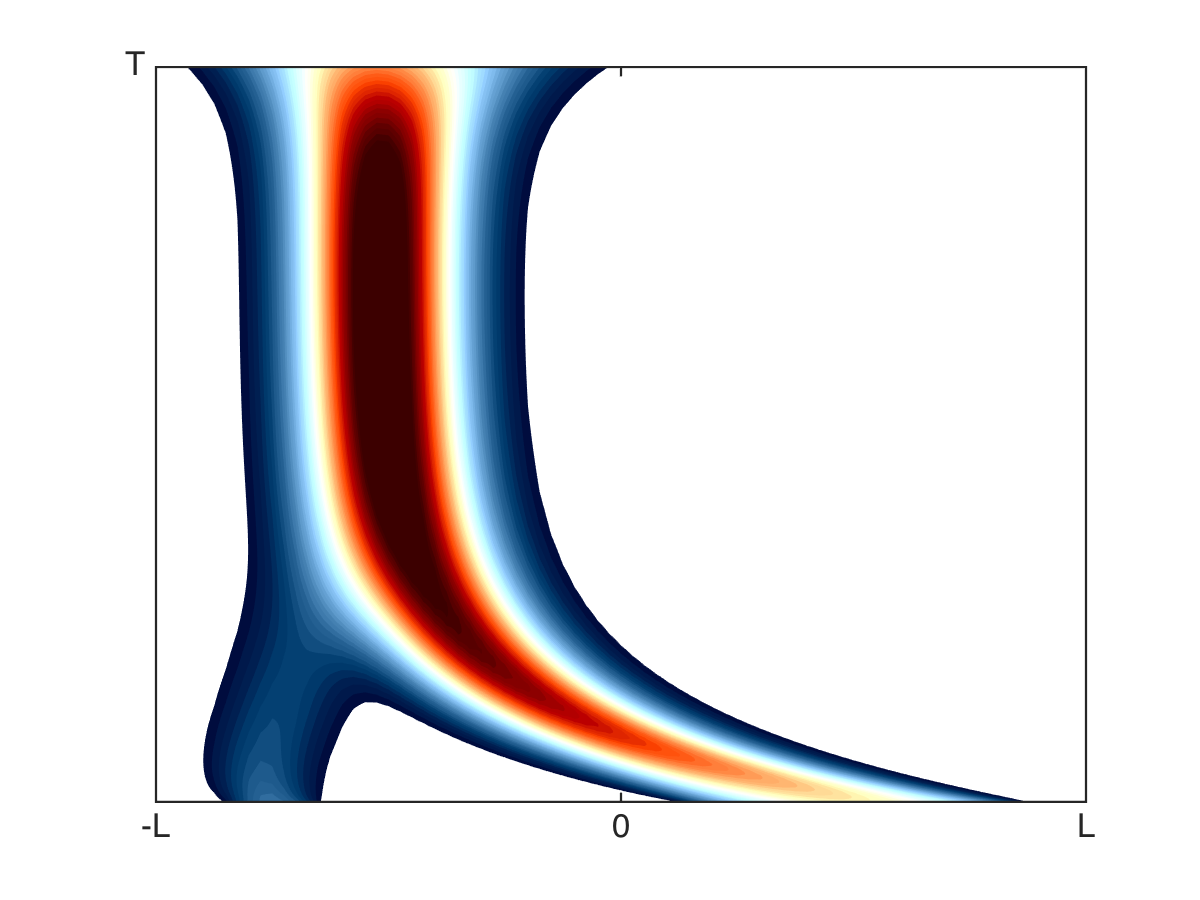}
\\
\sidecap{$f(x,t)$}
&
\includegraphics[trim=30 10 40 20,clip,width=0.25\textwidth]{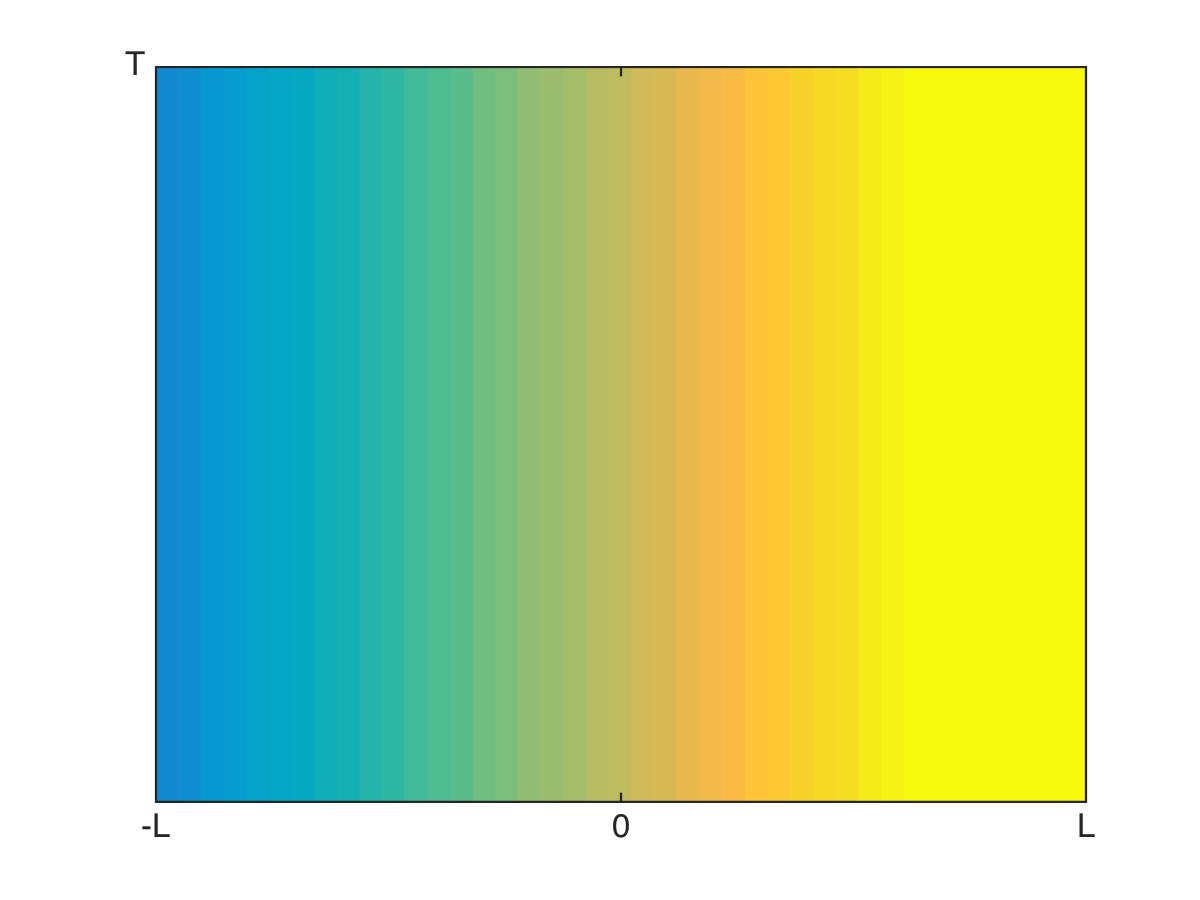}&
\includegraphics[trim=30 10 40 20,clip,width=0.25\textwidth]{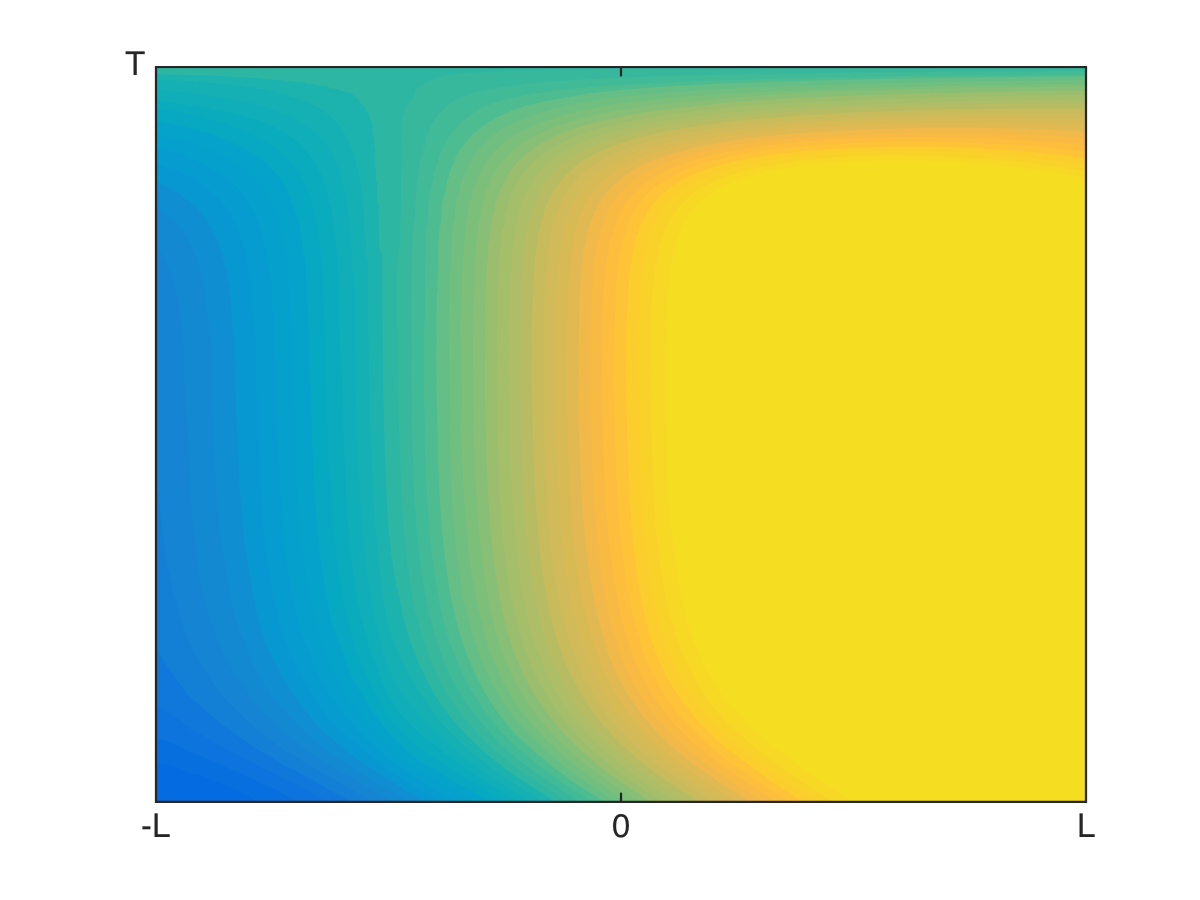}&
\includegraphics[trim=30 10 40 20,clip,width=0.25\textwidth]{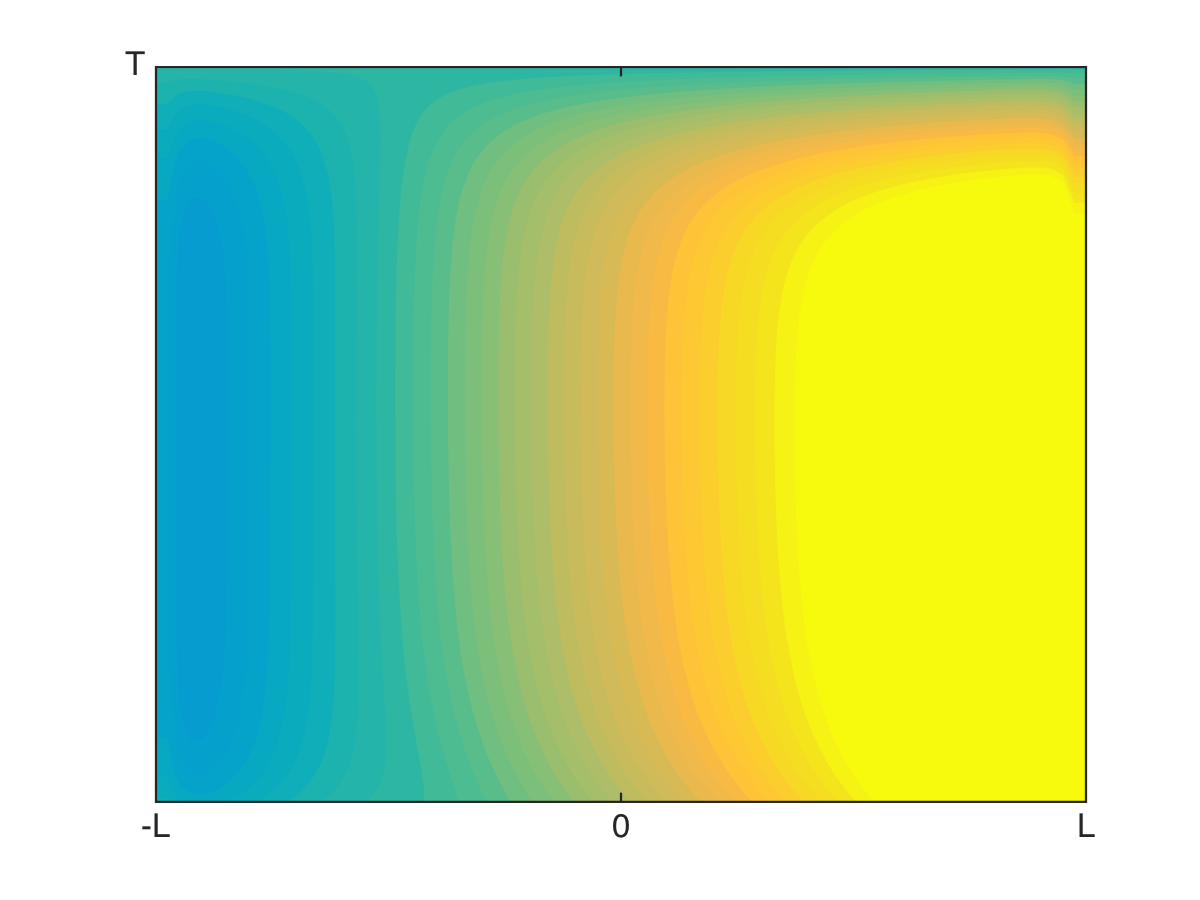}
\\
\hline
\hline
 $\gamma=0.05$&$J(\mu,f)=0.3648  $  & $J(\mu,f)=0.2835$ & $J(\mu,f)=0.2707$ \\
 \hline
\sidecap{$\mu(x,t)$} 
&
\includegraphics[trim=30 10 40 20,clip,width=0.25\textwidth]{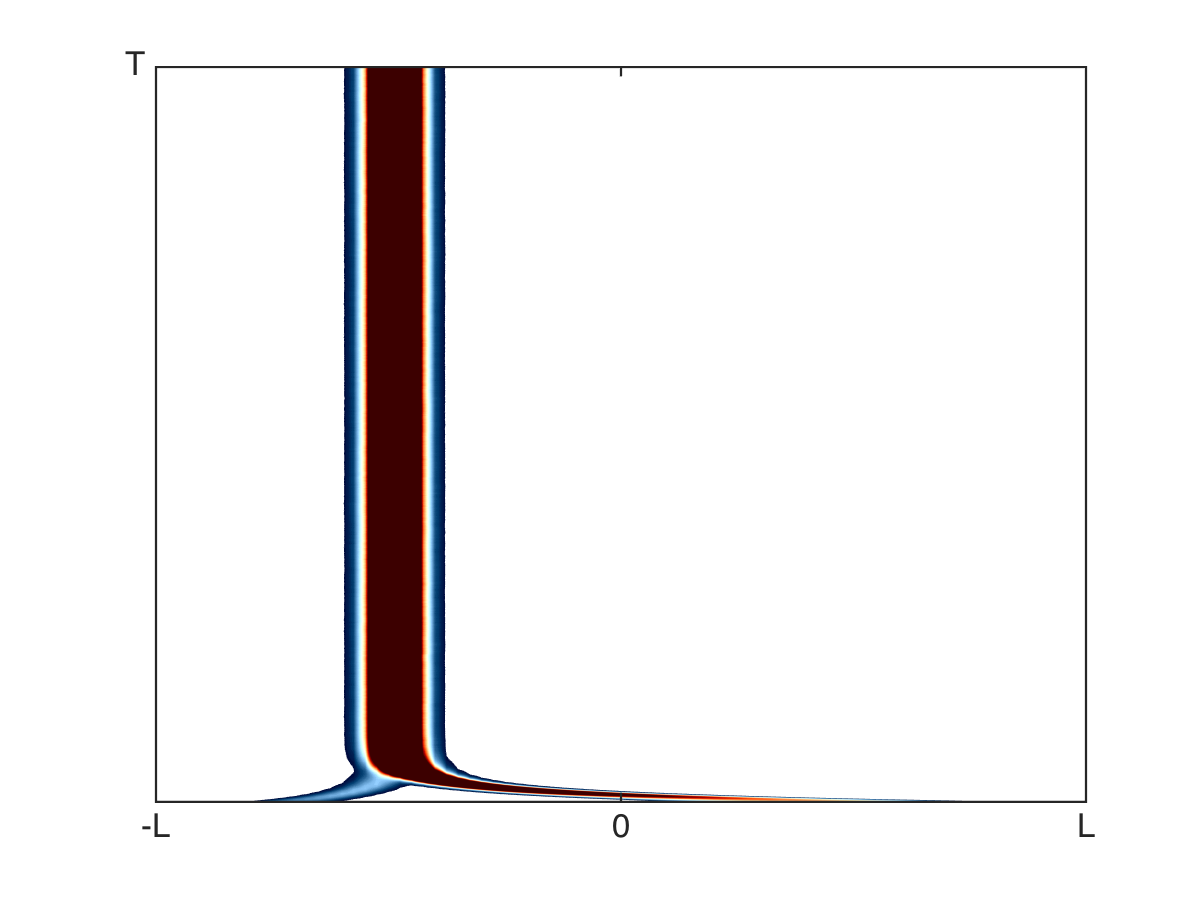}&
\includegraphics[trim=30 10 40 20,clip,width=0.25\textwidth]{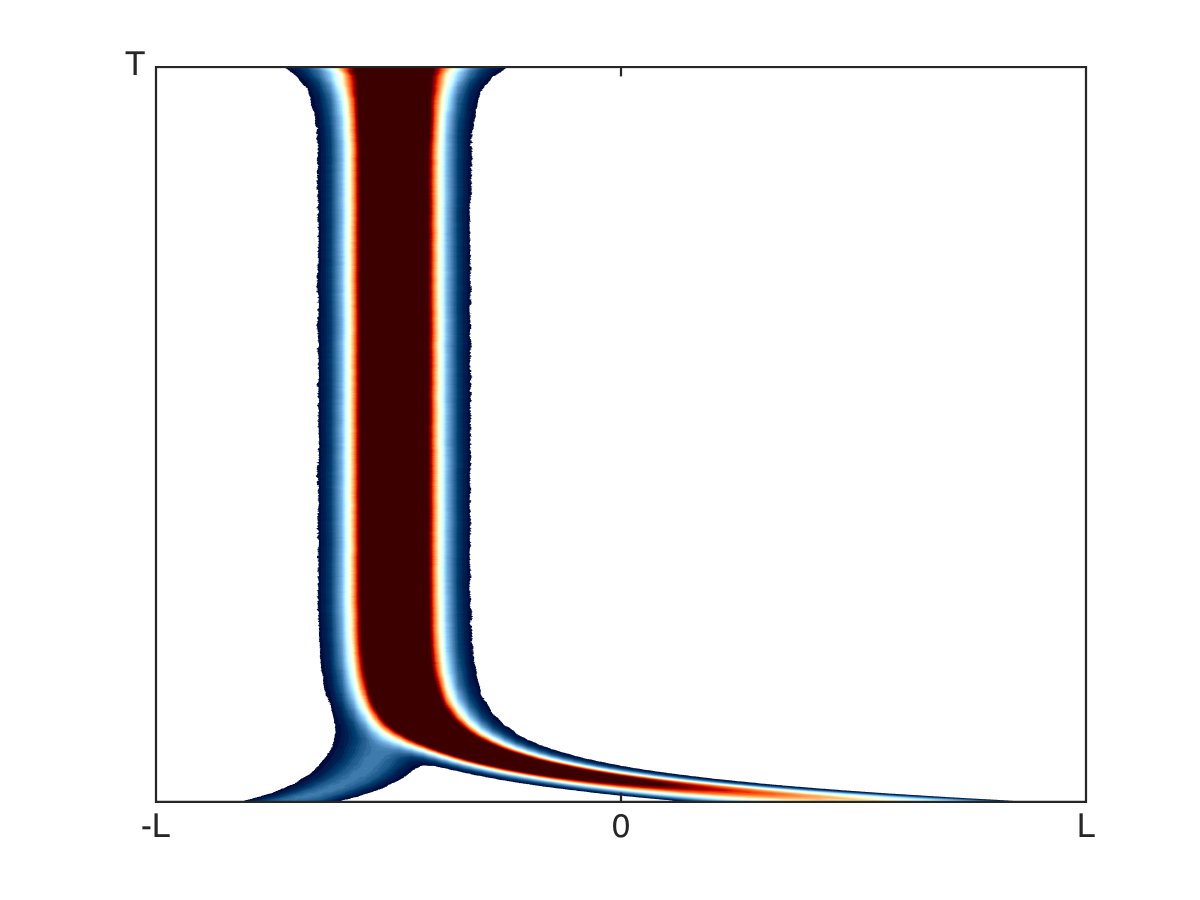}&
\includegraphics[trim=30 10 40 20,clip,width=0.25\textwidth]{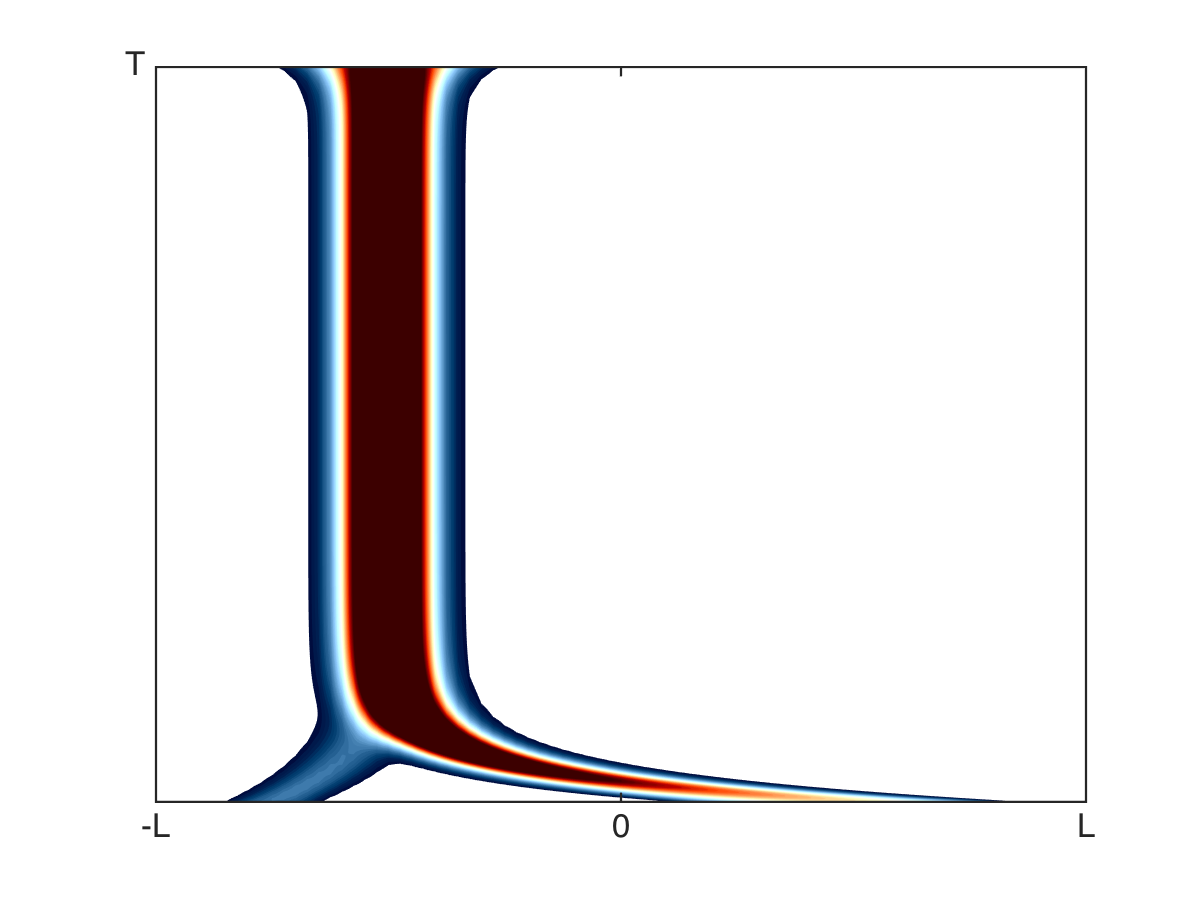}
\\
\sidecap{$f(x,t)$}
&
\includegraphics[trim=30 10 40 20,clip,width=0.25\textwidth]{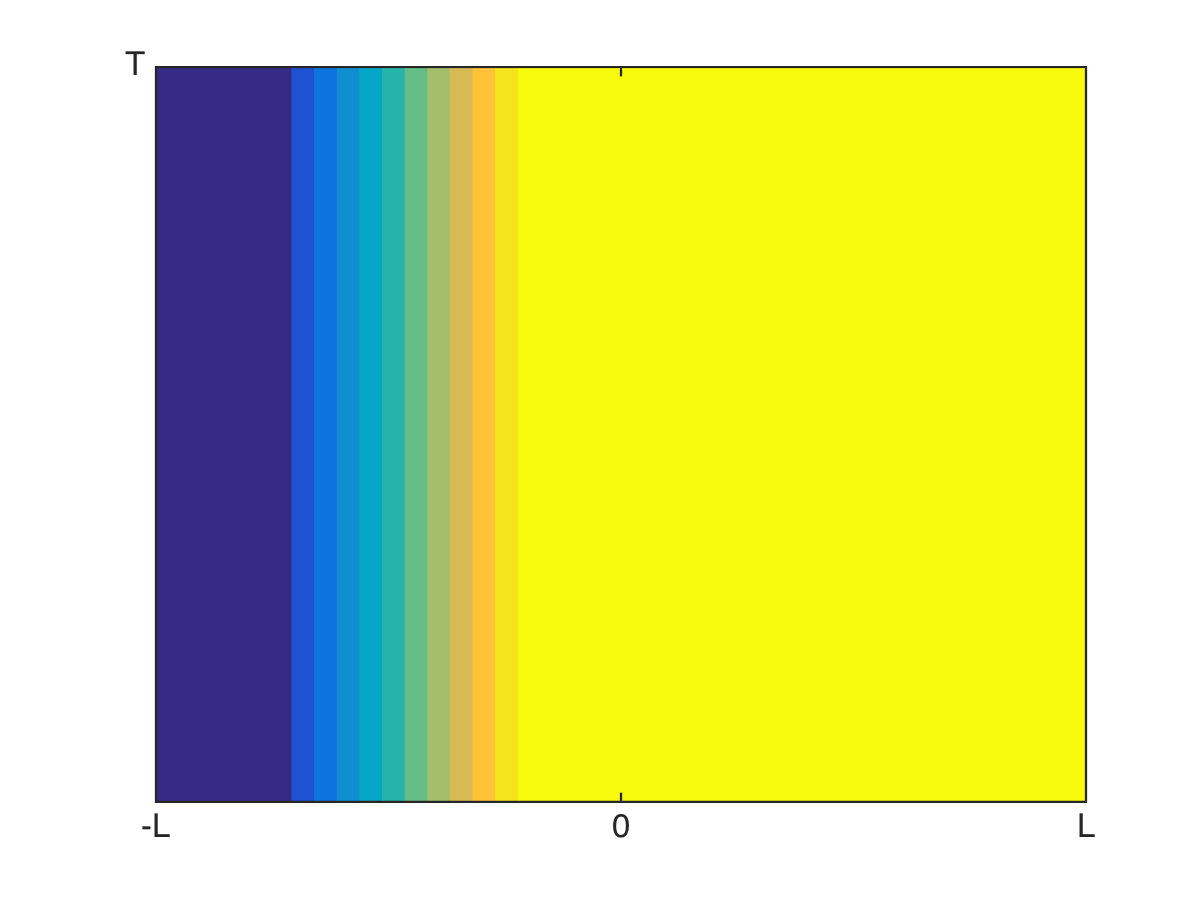}&
\includegraphics[trim=30 10 40 20,clip,width=0.25\textwidth]{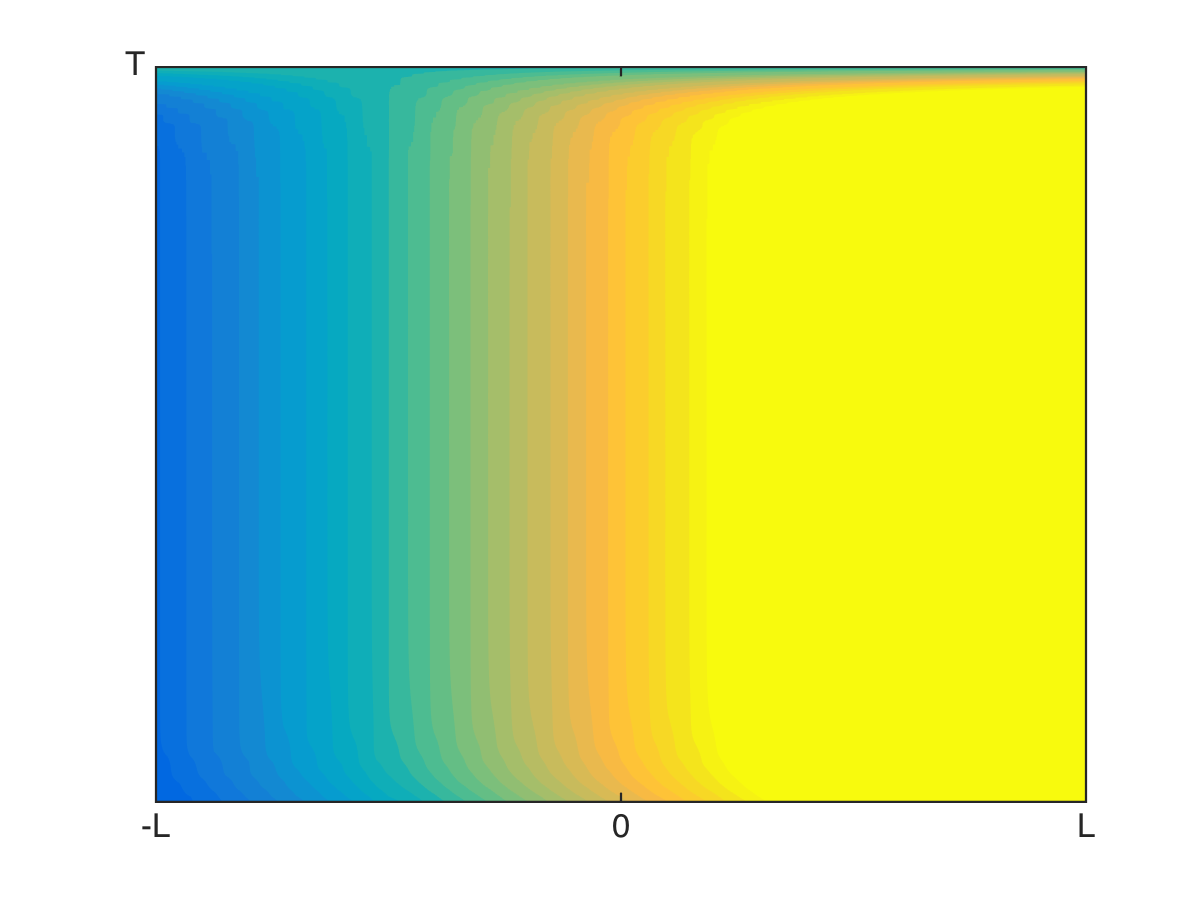}&
\includegraphics[trim=30 10 40 20,clip,width=0.25\textwidth]{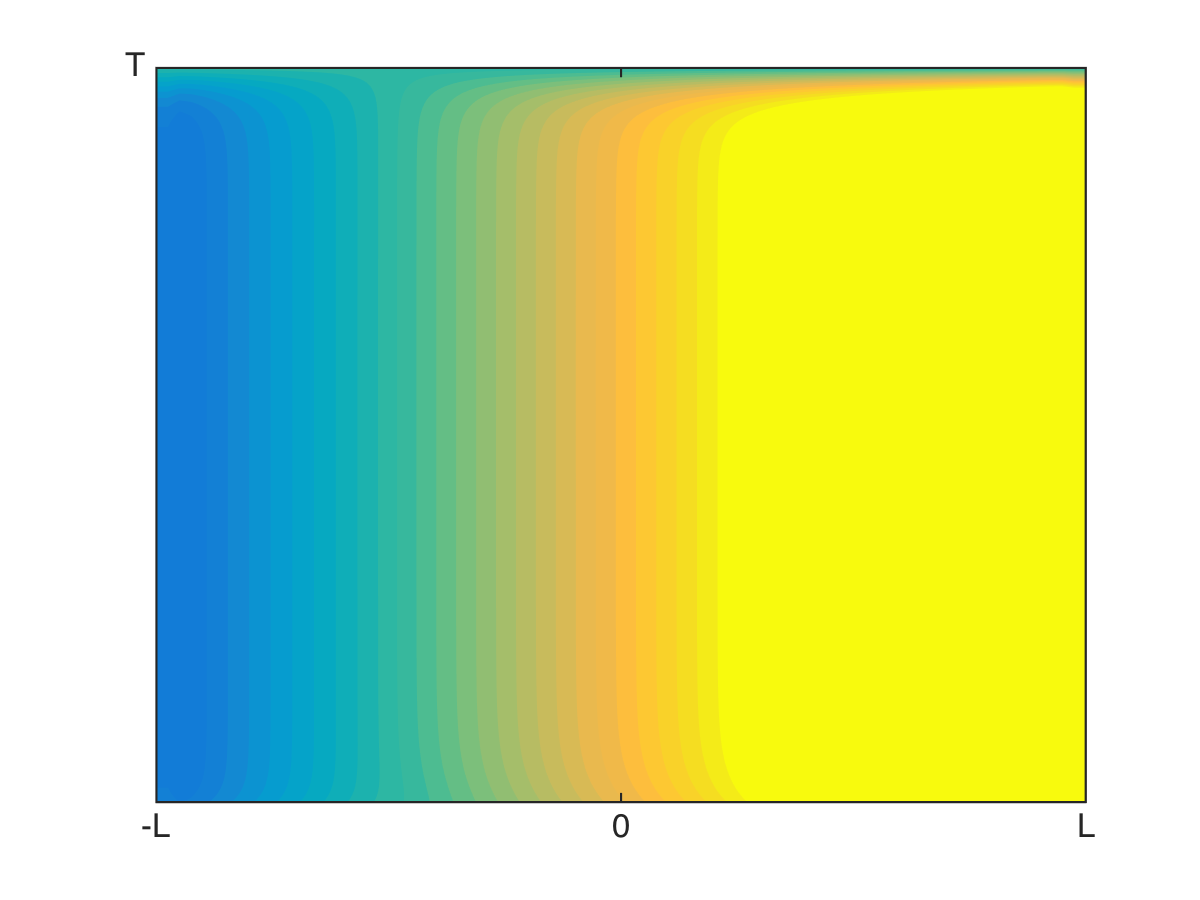}
\\
\hline
\end{tabular}
\end{center}
\caption{Test \#1: Transient behavior of the density $\mu(x,t)$ and the control $f(x,t)$ in $[-L,+L]\times[0,T]$, with $L =1,~T=8$,  for the Sdnajz's model, \eqref{eq:MFC}-\eqref{eq:MFSz}. The top picture depicts the transient density of the unconstrained dynamics. Value of the cost functional are reported in correspondence of the choice of the method and the penalization parameter $\gamma $.}\label{Fig_T1b}
\end{figure}

\subsubsection{Test 2: Hegselmann-Krause model}
In this second test we consider the mean field Hegselmann-Krause model \cite{hekr02}, also known as bounded confidence model, whose interaction kernel reads
\begin{align}\label{eq:MFHK}
P(x,y)=\chi_{\{|x-y|\leq\kappa\}}(y).
 \end{align}
This type of model describes the propensity of agents  to interact only within a confidence range $K=[x-\kappa,x+\kappa]$ of their opinion $x$, in the present experiment we fix $\kappa = 0.15$.
Thus we study the evolution of the control problem \eqref{eq:MFC}--\eqref{eq:J_MFC} up to time $T = 20$ with initial data defined as
$\mu^0(x) = C_0(0.5+\epsilon(1-x^2)),$ for $\epsilon = 0.01$ and $C_0$ such that the total density is a probability distribution. 
The diffusion coefficient is $\sigma = 10^{-5}$, the penalization parameter $\gamma = 2.5$, and  the desired state $x_d = 0$.

The uncontrolled evolution of this model shows the emergence of multiple clusters, as it is shown in the top picture of Figure \ref{Fig_T2}, due to the small value of $\kappa$ and small diffusion. Figure \ref{Fig_T2} depicts the transient behavior of the density $\mu(x,t)$ and the control signal $f(x,t)$ in the frame $\Omega\times[0,T]$. 

We observe in Figure \ref{Fig_T2} that for the instantaneous control (IC), consensus is slowly reached with a cost functional value of $J_{IC}(\mu,f)=0.8807$; the finite horizon control (FH) and the solution of the optimality conditions (OC) are able to steer faster the system towards $x_d$, respectively with cost $J_{FH}(\mu,f)= 0.6079$, and $J_{OC}(\mu,f)= 0.5570$.
\\

\begin{figure}[!ht]
\begin{center}
\begin{tabular}{@{}c@{\hspace{1mm}}c@{\hspace{1mm}}c@{\hspace{1mm}}c@{}}
~&&\textrm{uncontrolled}&\\
~&&\includegraphics[trim=30 10 40 20,clip,width=0.25\textwidth]{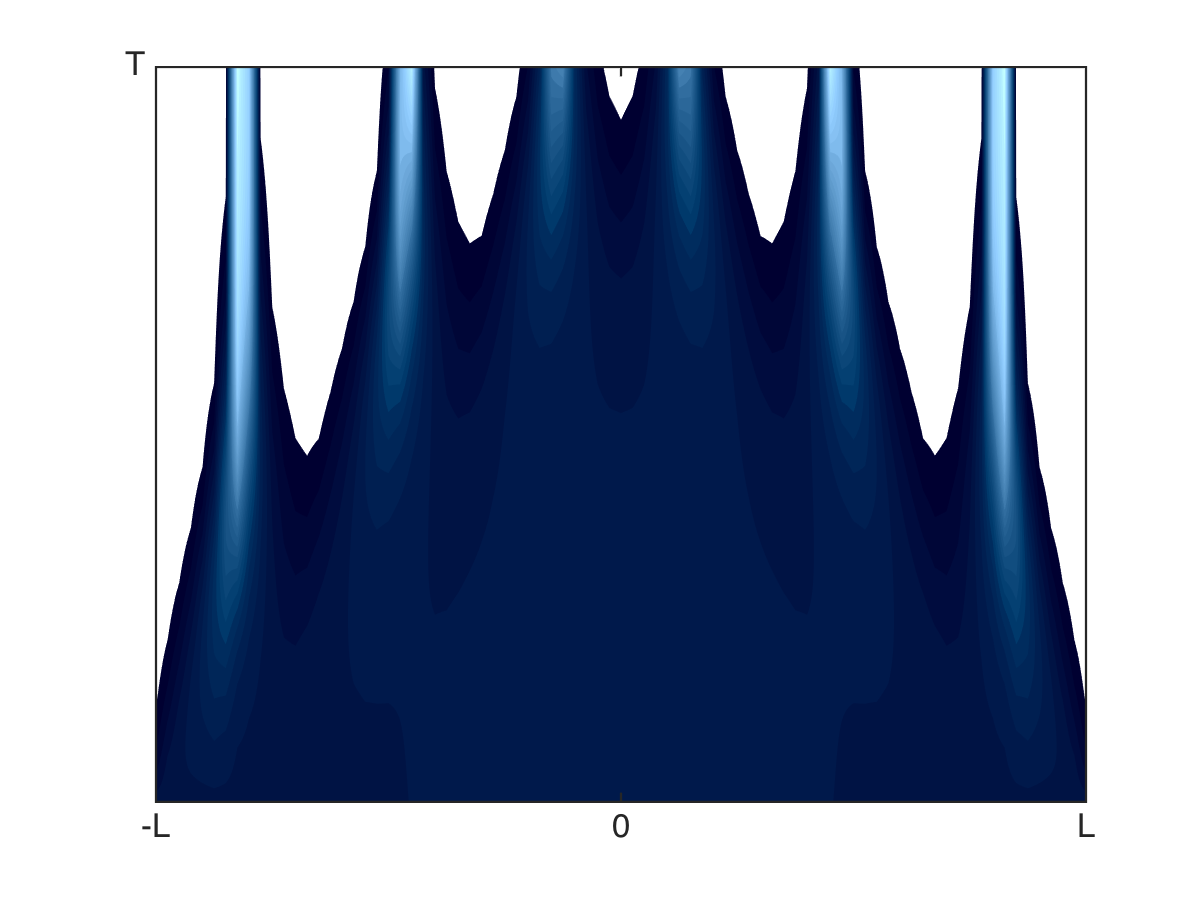}&\\
\hline
&$IC$ & $FH$ & $OC$\\
\hline
\hline
 $\gamma=2.5$&$J(\mu,f)=0.8807 $  & $J(\mu,f)=0.6079$ & $J(\mu,f)=0.5570$ \\
 \hline
\sidecap{$\mu(x,t)$} 
&
\includegraphics[trim=30 10 40 20,clip,width=0.25\textwidth]{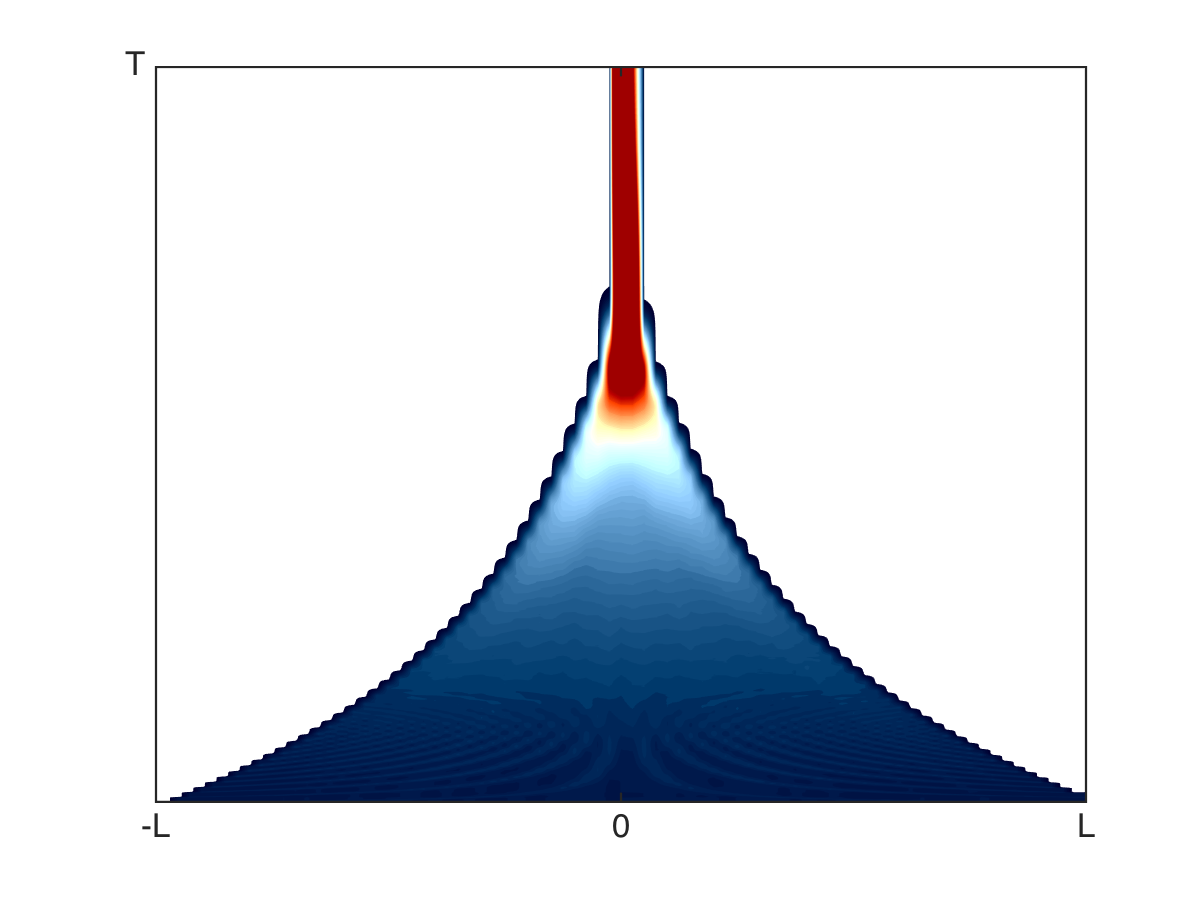}&
\includegraphics[trim=30 10 40 20,clip,width=0.25\textwidth]{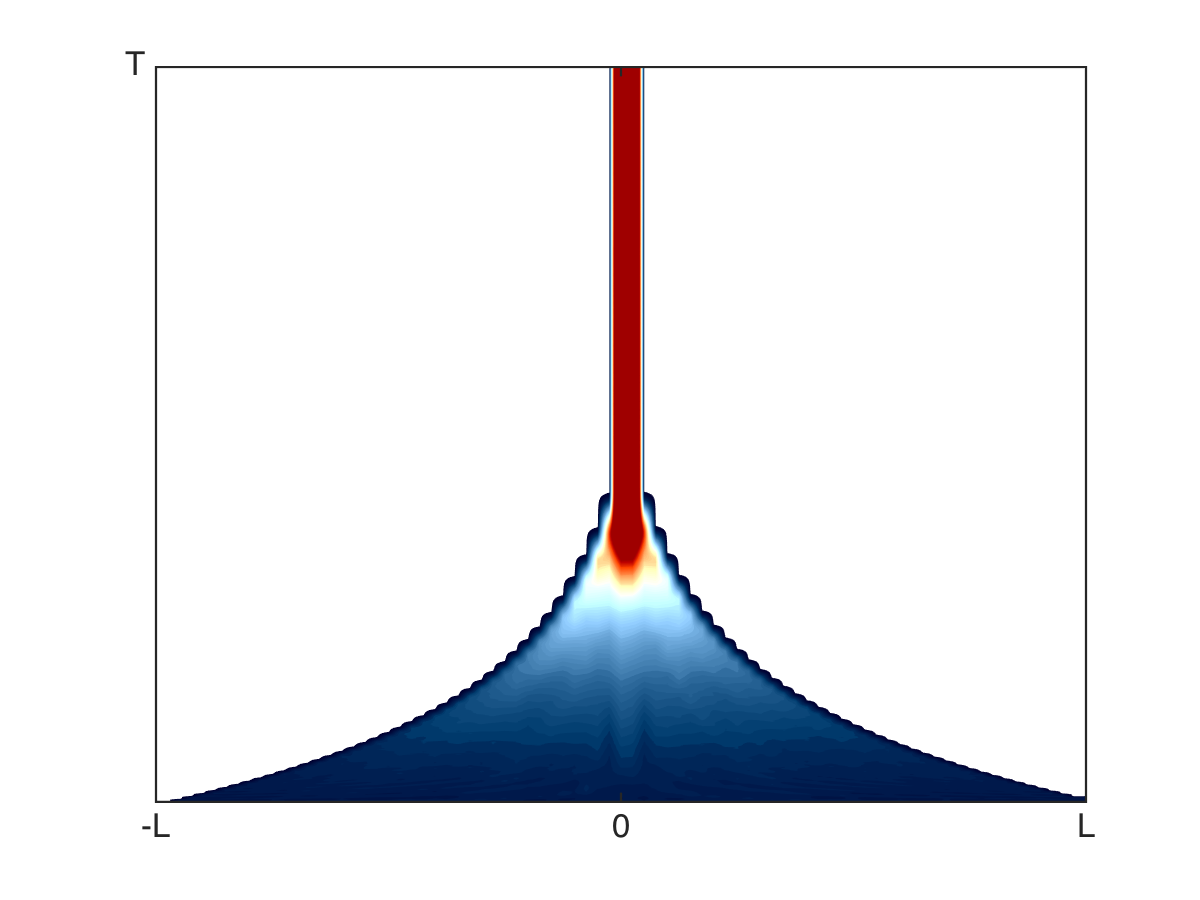}&
\includegraphics[trim=30 10 40 20,clip,width=0.25\textwidth]{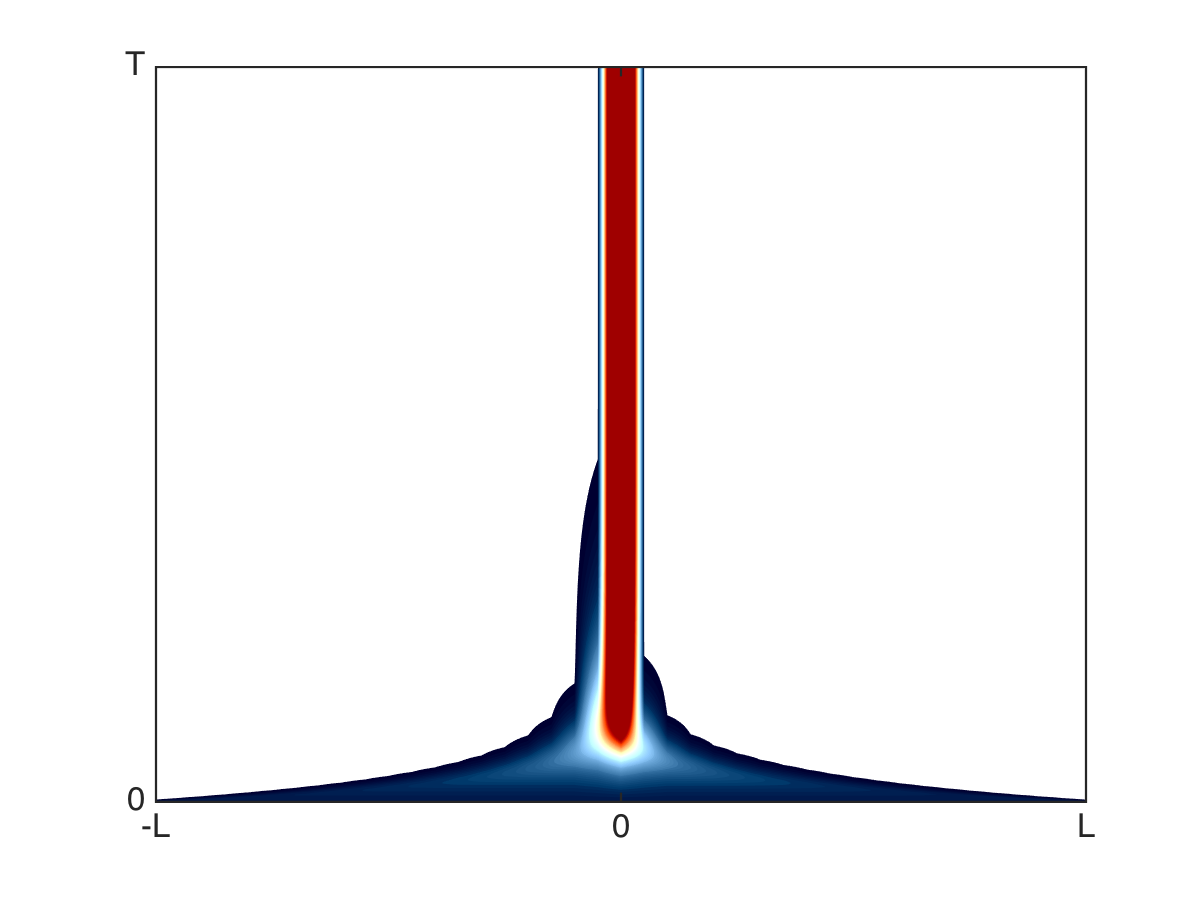}
\\
\sidecap{$f(x,t)$}
&
\includegraphics[trim=30 10 40 20,clip,width=0.25\textwidth]{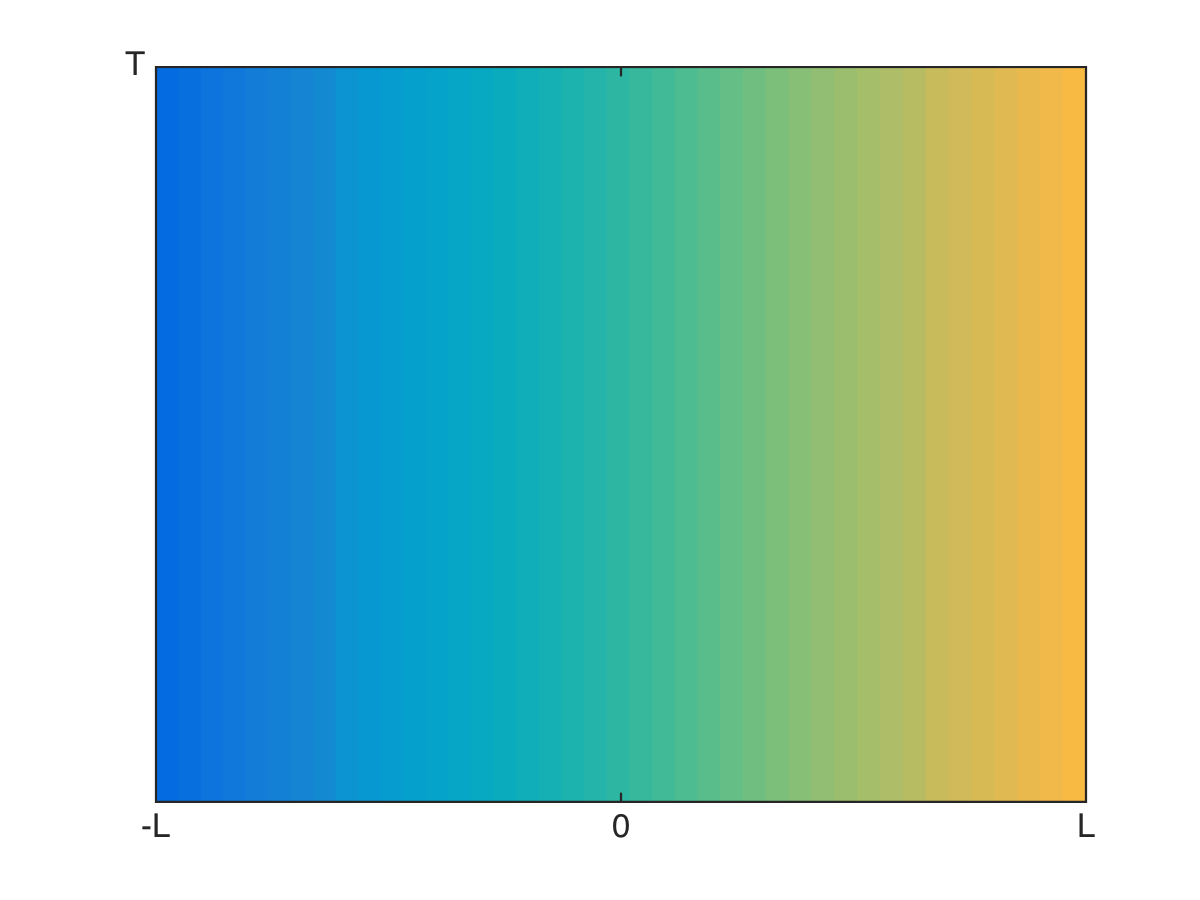}&
\includegraphics[trim=30 10 40 20,clip,width=0.25\textwidth]{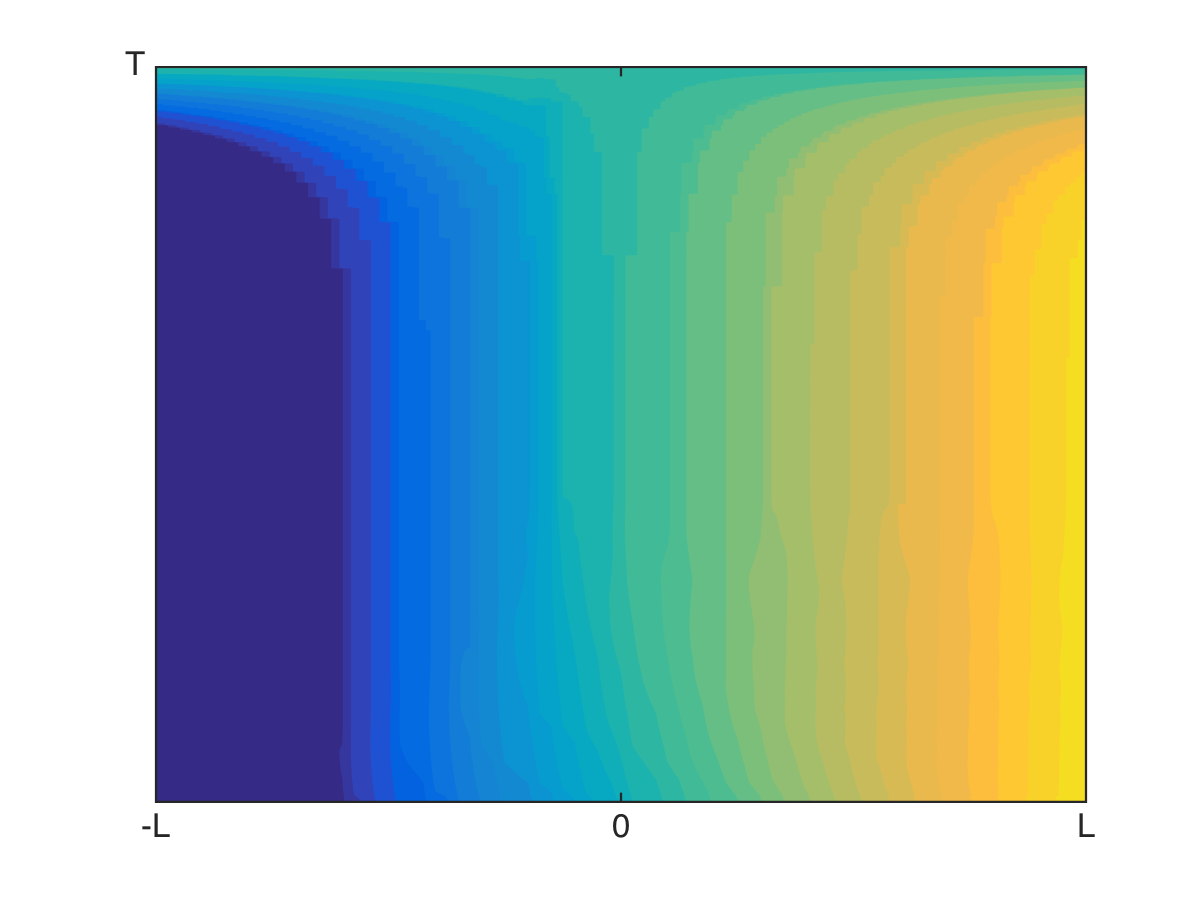}&
\includegraphics[trim=30 10 40 20,clip,width=0.25\textwidth]{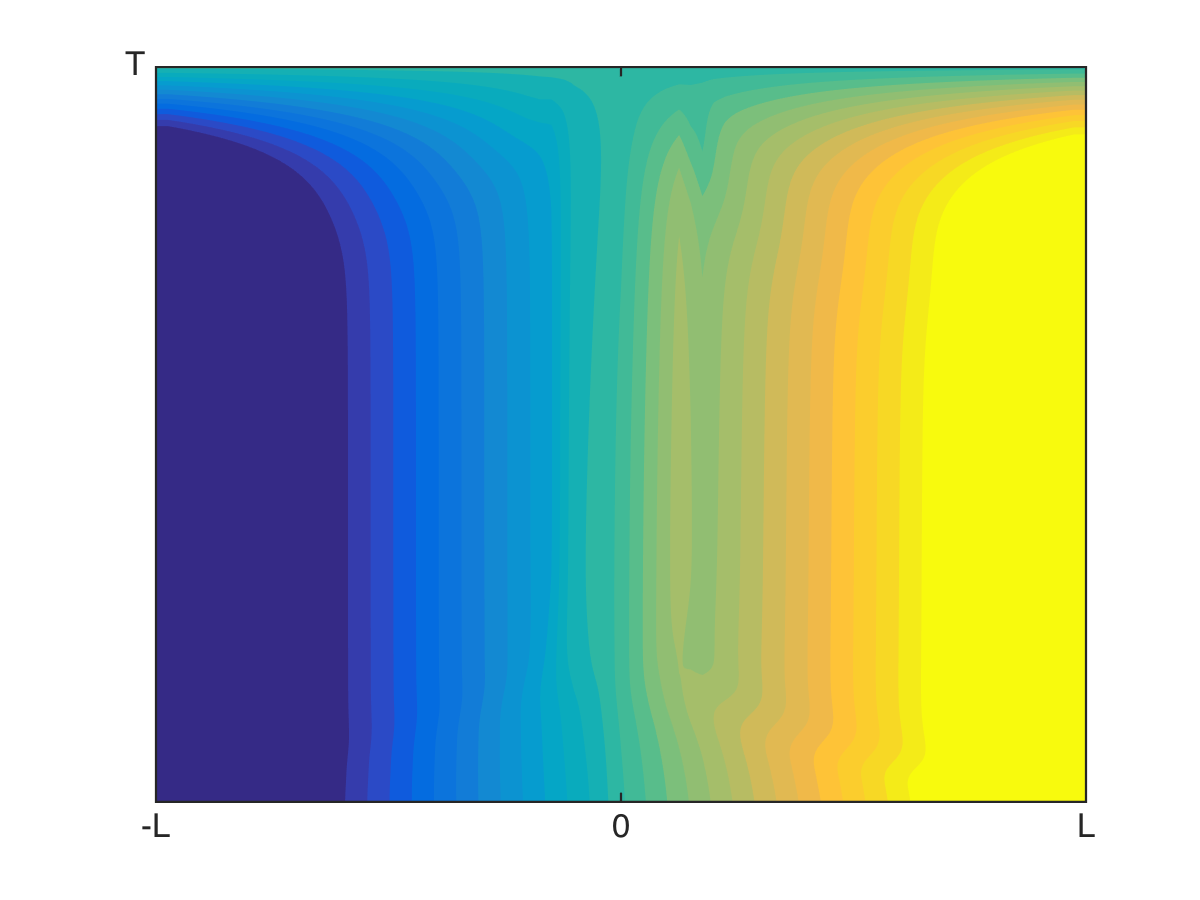}
\\
\hline
\end{tabular}
\end{center}
\caption{Test \#2: Transient behavior of the density $\mu(x,t)$ and the control $f(x,t)$ in $[-L,+L]\times[0,T]$, with $L =1,~T=20$,  for the Hegelmann-Krause's model, \eqref{eq:MFC}-\eqref{eq:J_MFC}. The top picture shows the emergence of opinion clustering in the unconstrained dynamics. Value of the cost functional are reported in correspondence of the choice of the method with penalization parameter $\gamma = 2.5$.}\label{Fig_T2}
\end{figure}

These experiments are showing very clearly the hierarchy of the controls (IC)$\rightarrow$(FH)$\rightarrow$(OC). In particular, it is evident the quasi-optimality of (FH), to the extent that we can claim (FH) $\approx$ (OC). The intuition is that (FH) is an optimal control on the binary dynamics of two particles, and, through the Boltzmann collisional operator, its binary optimality is ``smeared'' on the entire population. However, we have no quantitative method yet to assess such an approximation.
In fact, as commented in Remark \ref{quasiopt}, although the (FH) fulfills a Hamilton-Jacobi-Bellman equation, its synthesis by means of \eqref{eq:kernelK} to control \eqref{eq:FP} unfortunately does not fulfill \eqref{eq:backward}, even not approximately: by testing  \eqref{eq:kernelK} within  \eqref{eq:backward}, there a few useful cancelations, but, because of lack of symmetry, certain terms remains, whose magnitude is still hard to estimate. We expect that those terms are actually not so large and this would somehow justify the quasi-optimality of (FH). This issue remains an interesting open problem.

\paragraph{Concluding remarks.} In this paper, we have presented a hierarchy of control designs for mean field dynamics. At the bottom of the hierarchy, we have introduced optimal feedback controls which are derived for two-agent models, and which are subsequently realized at the mean field level through a Boltzmann approach. At the top of the hierarchy, one finds the mean field optimal control problem and its correspondent optimality conditions. In both cases, we presented a theoretical and numerical analysis of the proposed designs, as well as computational implementations. From the numerical experiments presented in the last section, we observe that although the numerical realization of the mean field optimality system yields the best controller in terms of the cost functional value, feedback controllers obtained for the binary system perform reasonably well, and provide a much simpler control synthesis. We expect to further proceed along this direction of research, in particular in relation to the computation of feedback controllers via Dynamic Programming and Hamilton-Jacobi-Bellman equations for the binary system, as it provides a versatile framework to address different control problems.

\paragraph*{Acknowledgements.} {GA, YPC, and MF acknowledge the support of the ERC-Starting Grant HDSPCONTR "High-Dimensional Sparse Optimal Control". YPC is also supported by the Alexander Humboldt Foundation through the Humboldt Research Fellowship for Postdoctoral Researchers. DK acknowledges the support of the ERC-Advanced Grant OCLOC "From Open-Loop to Closed-Loop Optimal Control of PDEs".}

\end{document}